\documentclass[a4paper, 11pt, twoside, notitlepage]{amsart}

\usepackage[utf8]{inputenc}
\usepackage{color}
\usepackage{amsmath} 
\usepackage{amssymb} 
\usepackage{amsthm}
\usepackage{geometry}
\usepackage{graphicx}
\usepackage{esint}
\usepackage[colorlinks=true,linkcolor=blue]{hyperref}

\theoremstyle{plain}
\newtheorem{thm}{Theorem}
\newtheorem*{thm*}{Theorem}
\newtheorem{prop}{Proposition}[section]
\newtheorem{lem}[prop]{Lemma}
\newtheorem{cor}[prop]{Corollary}

\newtheorem{defi}[prop]{Definition}
\newtheorem{rmk}[prop]{Remark}

\newcommand {\R} {\mathbb{R}} 
 \newcommand {\N} {\mathbb{N}}
\newcommand {\C} {\mathbb{C}} 
\newcommand {\p} {\partial}

\newcommand {\D} {\Delta}

\newcommand {\supp} {\text{supp}}

\newcommand{\ov}{\overline{\varphi}}

\newcommand{\ou}{\overline{u}}

\newcommand{\mR}{\mathbb{R}}                    
\newcommand{\abs}[1]{\lvert #1 \rvert}          
\newcommand{\norm}[1]{\lVert #1 \rVert}         

\DeclareMathOperator {\dist} {dist}

\DeclareMathOperator{\F} {\mathcal{F}}
\DeclareMathOperator{\Je} {\mathcal{J}_{\epsilon,h,s}}

\DeclareMathOperator{\Jes} {\mathcal{J}_{\epsilon,h,s,\delta}}

\title[Quantitative Approximation]{Quantitative Approximation Properties for the Fractional Heat Equation}

\author[A. R\"uland]{Angkana R\"uland}
\address{Mathematical Institute, University of Oxford, Andrew Wiles Building, Radcliffe Observatory Quarter, Woodstock Road, Oxford OX2 6GG}
\email{ruland@maths.ox.ac.uk}

\author[M. Salo]{Mikko Salo}
\address{Department of Mathematics and Statistics, University of Jyv\"askyl\"a}
\email{mikko.j.salo@jyu.fi}

\begin{document}

\maketitle

\begin{abstract}
In this note we analyse \emph{quantitative} approximation properties of a certain class of \emph{nonlocal} equations: Viewing the fractional heat equation as a model problem, which involves both \emph{local} and \emph{nonlocal} pseudodifferential operators,
we study quantitative approximation properties of solutions to it. First, relying on Runge type arguments, we give an alternative proof of certain \emph{qualitative} approximation results from \cite{DSV16}. Using propagation of smallness arguments, we then provide bounds on the \emph{cost} of approximate controllability and thus quantify the approximation properties of solutions to the fractional heat equation. Finally, we discuss generalizations of these results to a larger class of operators involving both local and nonlocal contributions.
\end{abstract}

\section{Introduction}

This article is dedicated to \emph{qualitative} and \emph{quantitative} approximation properties of solutions to certain mixed \emph{local-nonlocal} equations. As a model problem, we consider the heat equation for the fractional Laplacian with $s\in(0,1)$,
\begin{equation}
\label{eq:eq_main}
\begin{split}
(\p_t +(-\D)^{s}) u &= 0 \mbox{ in } B_1 \times (-1,1),\\
 u & = f \mbox{ in } (\R^n \setminus \overline{B}_1) \times (-1,1) ,\\
 u & = f \hspace{0pt} \mbox{ in } \R^n \times \{-1\},
\end{split}
\end{equation}
and study the \emph{quantitative} approximation properties of the mapping 
\begin{align}
\label{eq:map}
L^2((-1,1), C^{\infty}_c(W)) \ni f \mapsto P_sf=u|_{B_1 \times (-1,1)} \in L^2(B_1 \times (-1,1)).
\end{align}
Here $W\subset \R^n$ is an open, bounded Lipschitz set, such that
$\overline{W}\cap \overline{B}_1 = \emptyset$. The precise definition of the solution map $P_s$ for the problem \eqref{eq:eq_main} is given in Section \ref{sec:qual}.

Due to the work of Dipierro, Savin and Valdinoci \cite{DSV16} (c.f.\ also \cite{DSV14}), it is known that the mapping $P_s$ has a dense image (even in suitable H\"older spaces). More precisely, the authors show the following result:

\begin{thm*}[\cite{DSV16}, Theorem 1]
Let $B_1 \subset \R^n$ be the unit ball, $s\in (0,1)$, $k\in \N$ and $h:B_1 \times (-1,1) \rightarrow \R$ with $h\in C^k(\overline{B}_1 \times [-1,1])$. Fix $\epsilon>0$. Then there exists $u_{\epsilon}=u\in C^{\infty}(B_1 \times (-1,1))\cap C(\R^{n+1})$ which is compactly supported in $\R^{n+1}$ and such that the following properties hold true:
\begin{align*}
&\p_t u + (-\D)^s u  = 0 \mbox{ in } B_1 \times (-1,1),\\
& \mbox{ and } \|u-h\|_{C^k(B_1 \times (-1,1))} \leq \epsilon.
\end{align*} 
\end{thm*}

Moreover, the results in \cite{DSV16} show that the approximation properties of \eqref{eq:eq_main} are determined by the nonlocal part of the operator. In particular, in the framework of \cite{DSV16} parabolicity is not needed, it would for instance also be possible to consider wave type operators.

While showing the density of the image of the mapping \eqref{eq:map}, the argument in \cite{DSV16} does not quantify the \emph{cost} of approximating a given function $h\in L^2(B_1 \times (-1,1))$. In this note we address this question: 
\begin{itemize}
\item[Q:]
\emph{Given an error threshold $\epsilon>0$ and a function $h\in L^2(B_1 \times (-1,1))$, how large is the value of a suitable norm of a possible control function $f$, which is such that $P_sf$ approximates $h$ up to the error threshold $\epsilon$?}
\end{itemize}
These quantitative arguments 
were partly motivated by stability results in inverse problems for
nonlocal operators, c.f. \cite{RS17}, and can also be considered as a continuation of the investigation started in \cite{Rue17}.
In the context of the model problem \eqref{eq:eq_main} our main result can be formulated as the following proposition:

\begin{thm}[Cost of approximation]
\label{prop:cost}
Let $h\in H^1_0(B_1 \times (-1,1))$ and $\epsilon>0$. Let $W \subset \R^n \setminus \overline{B}_1$ be a Lipschitz domain with $\overline{W} \cap \overline{B}_1 = \emptyset$.
Then there exists a control function $f \in L^2((-1,1), C^{\infty}_c(W))$ such that
\begin{equation}
\label{eq:approx_cost}
\begin{split}
&\|h-P_s f\|_{L^2(B_1 \times (-1,1))} \leq \epsilon,\\ 
&\|f\|_{L^2(W \times (-1,1))}\leq C e^{C (1+\|h\|_{H^1(B_1 \times (-1,1))}^{\sigma})\epsilon^{-\sigma}}\|h\|_{H^1(B_1 \times (-1,1))},
\end{split}
\end{equation}
where the constants $C > 1$ and $\sigma > 0$ only depend on $n$, $s$, and $W$. 
Moreover, we note that $f$ can be expressed in terms of the minimizer of a suitable ``energy" (more precisely of the functional \eqref{eq:functional}).  
\end{thm}

Let us comment on this result: In the model setting of the heat equation for the fractional Laplacian it quantifies an $L^2$ version of the result from \cite{DSV16}. The condition that $h$ vanishes on the boundary does not pose serious restrictions compared to the result of \cite{DSV16}, as this can always be achieved after a suitable extension. Indeed, it is always possible to reduce to the situation where $h \in H^1_0$ by considering the control problem in a slightly larger Lipschitz domain $\Omega \times (-2,2) \subset B_2 \times (-2,2)$ (where $\Omega$ is adapted to the geometry of $B_1$ and $W$) and by extending the given function $h\in H^1(B_1 \times (-1,1))$ to a function $\tilde{h} \in H^1_0(\Omega \times (-2,2))$ with the properties that $\tilde{h}|_{B_1 \times (-1,1)}= h$ and
\begin{align*}
\|\tilde{h}\|_{H^1(\Omega \times (-2,2))} \leq C(W,B_1)\|h\|_{H^1(B_1 \times (-1,1))}.
\end{align*}
Considering an analogue of \eqref{eq:eq_main} and Theorem \ref{prop:cost} in $\Omega \times (-2,2) $ then implies the $L^2$ version of the approximation result from \cite{DSV16} for the fractional heat equation.

Regarding the dependences on $\epsilon$ and $h$ in the estimate \eqref{eq:approx_cost} in Theorem \ref{prop:cost}, we expect that the exponential dependence on $\epsilon>0$ is indeed necessary. Although it is natural that higher order norms of $h$ appear in the estimate, we do not believe that the norms, which are used in \eqref{eq:approx_cost}, are optimal. Yet we hope that the ideas introduced here are robust enough to be extended to a number of other problems in which both local and nonlocal operators are involved. A number of further operators for which these ideas are applicable are discussed in Section \ref{sec:extend}. 

Similarly as in \cite{Rue17}, our approach to the question on the cost of control relies on
\begin{itemize}
\item[(i)] a propagation of smallness result, 
\item[(ii)] quantitative unique continuation properties of the adjoint equation \eqref{eq:dual}, 
\item[(iii)] the variational technique from \cite{FZ00},
\item[(iv)] and on a global estimate for solutions to \eqref{eq:dual} (c.f. equation \eqref{eq:global}). 
\end{itemize}
As in the qualitative density result, it is the underlying \emph{nonlocal} operator, whose properties we mainly exploit (c.f. ingredients (i)-(iii)). The parabolic character of the problem only enters by invoking global estimates. It is therefore possible to extend this result to a much richer class of local-nonlocal operators (c.f. Section \ref{sec:extend}).

The remainder of the article is structured as follows: In Section \ref{sec:qual} we first discuss the qualitative approximation properties of the fractional heat equation. This is based on Runge type approximation arguments. Next, in Section \ref{sec:quant}, we address the quantitative uniqueness properties for the fractional heat equation with $s\in(0,1)$. Here we rely on propagation of smallness estimates. In Section \ref{sec:proof_main} we introduce a variational approach to the approximation problem and prove Theorem \ref{prop:cost}. Finally, in Section \ref{sec:extend}, we explain how to extend the presented arguments to more general (variable coefficient) local-nonlocal operators.

\subsection*{Acknowledgements}

A.R.\ gratefully acknowledges a Junior Research Fellowship at Christ Church. M.S.\ was supported by the Academy of Finland (Finnish Centre of Excellence in Inverse Problems Research, grant number 284715) and an ERC Starting Grant (grant number 307023). Both authors would like to thank Herbert Koch for commenting on a preliminary version of Corollary \ref{cor:frac1} and suggesting a simplification in the argument.

\section{Qualitative Approximation and Weak Unique Continuation}
\label{sec:qual}
In this section, we discuss the \emph{qualitative} approximation properties of the mapping \eqref{eq:map}.
As the main result, we recover an $L^2$ version of certain approximation properties identified in \cite{DSV16}. Instead of relying on boundary asymptotics of the problem, we however use Runge type approximations as introduced in \cite{GSU16} (c.f.\ also \cite{L56}, \cite{B62}, \cite{B62a} for similar ideas in the setting of different local equations). In principle this could be upgraded to (stronger) approximation properties in H\"older spaces (c.f. Section 6 in \cite{GSU16} and \cite{DSV16}). As we are however mainly interested in the \emph{quantitative} approximation properties outlined in the next section, we do not pursue this here.

\subsection{Notation and well-posedness}
As in \cite{GSU16} and \cite{RS17} we will mainly use energy spaces. To that end, we recall that for $s\in \R$
\begin{align*}
&\widetilde{H}^s(B_1):= \mbox{ closure of $C_c^{\infty}(B_1)$ in $H^s(\R^n)$},\\
&H^s(B_1):= \{u|_{B_1}: u \in H^s(\R^n)\},
\end{align*}
and that 
\begin{align*}
(\widetilde{H}^s(B_1))^{\ast}= H^{-s}(B_1).
\end{align*}
We denote the corresponding homogeneous spaces by adding a dot to these spaces, e.g.\ $\dot{H}^s(\R^n)$.
As we are working with a time dependent problem, we will also use the corresponding Bochner spaces, which are associated with the energy spaces of our equations.

Having introduced the previous notation, we discuss the well-posedness of equations as in \eqref{eq:eq_main}. Here we restrict our attention to standard regularity assertions in the energy space as this suffices for our purposes. For more refined results in e.g. H\"older spaces we refer to \cite{FK13}, \cite{KS13}. We remark that the operator $(-\Delta)^s$ is always understood to act in the variable $x \in \R^n$.

\begin{lem}
\label{lem:well-posedness}
Let $n\geq 1$ and $s\in (0,1)$. Then for any $F\in L^2((-1,1),H^{-s}(B_1))$ and any $f\in L^2((-1,1), H^s(\R^n))$ with $f|_{B_1 \times (-1,1)} = 0$, there exists a unique function $u = f + v$, where $v \in L^2((-1,1),\tilde{H}^{s}(B_1))\cap C([-1,1],L^2(\R^n))$, satisfying
\begin{equation}
\label{eq:weaktwo}
\begin{split}
(\p_t +(-\D)^{s}) u &= F \mbox{ in } B_1 \times (-1,1),\\
 u & = f \mbox{ in } (\R^n \setminus \overline{B}_1) \times (-1,1) ,\\
 v & = 0 \hspace{5pt} \mbox{ in } \R^n \times \{-1\}.
\end{split}
\end{equation}
Moreover,
\begin{align*}
&\|u(t)\|_{L^2(B_1)} + \|u\|_{L^2((-1,1), H^s(\R^n))}
+ \|u'\|_{L^2((-1,1), H^{-s}(B_1))}\\
& \leq C(\|F\|_{L^2((-1,1), H^{-s}(B_1))} + \|f\|_{L^2((-1,1), H^s(\R^n))}).
\end{align*}
\end{lem}

\begin{rmk}
\label{rmk:weak_sol}
We refer to the function $u$ as a weak solution of \eqref{eq:weaktwo}. Note also that changing $t$ to $-t$, we obtain an analogous solvability result for the problem 
\begin{equation}
\label{eq:weaktwo_dual}
\begin{split}
(-\p_t +(-\D)^{s}) u &= F \mbox{ in } B_1 \times (-1,1),\\
 u & = f \mbox{ in } (\R^n \setminus \overline{B}_1) \times (-1,1) ,\\
 v & = 0 \hspace{5pt} \mbox{ in } \R^n \times \{1\}.
\end{split}
\end{equation}
\end{rmk}

\begin{proof}
We first note that writing $v=u-f$ and invoking the support assumption for $f$, the problem reduces to finding $v$ solving 
\begin{equation}
\label{eq:weakthree}
\begin{split}
(\p_t +(-\D)^{s}) v &= \tilde{F} \mbox{ in } B_1 \times (-1,1),\\
 v & = 0 \hspace{5pt} \mbox{ in } \R^n \times \{-1\}, \\
v &\in L^2((-1,1),\tilde{H}^{s}(B_1))\cap C([-1,1],L^2(\R^n)),
\end{split}
\end{equation}
where $\tilde{F} := F - (-\Delta)^s f$ is another function in $L^2((-1,1), H^{-s}(B_1))$. Now if $v$ is such a function solving \eqref{eq:weakthree}, then multiplying the equation by $v$, integrating over $(-1,t) \times \R^n$, and using that $v(-1) = 0$ gives the initial estimate 
\[
\frac{1}{2} \norm{v(t)}_{L^2(B_1)}^2 + \norm{v}_{L^2((-1,1), \dot{H}^s(\R^n))}^2 \leq \norm{\tilde{F}}_{L^2((-1,1), H^{-s}(B_1))} \norm{v}_{L^2((-1,1), \tilde{H}^s(B_1))}.
\]
The Hardy-Littlewood-Sobolev inequality gives $\norm{w}_{L^2(B_1)} \leq C \norm{w}_{L^{\frac{2n}{n-2s}}} \leq C_{n,s} \norm{w}_{\dot{H}^s}$ for $w \in \tilde{H}^s(B_1)$ (if $n=1$ and $s \geq 1/2$, one can interpolate the easy $L^2 \to L^2$ and $\dot{H}^1 \to L^{\infty}$ bounds). Using this and Young's inequality yields that 
\[
\norm{v(t)}_{L^2(B_1)} + \norm{v}_{L^2((-1,1), H^s(\R^n))} \leq C_{n,s} \norm{\tilde{F}}_{L^2((-1,1), H^{-s}(B_1))},
\]
and using the equation once more implies the energy estimate 
\begin{multline}
\sup_{t \in (-1,1)} \norm{v(t)}_{L^2(B_1)} + \norm{v}_{L^2((-1,1), H^s(\R^n))} + \norm{\partial_t v}_{L^2((-1,1), H^{-s}(B_1))} \\
 \leq C_{n,s} \norm{\tilde{F}}_{L^2((-1,1), H^{-s}(B_1))} \label{heat_energy_basic}
\end{multline}
for solutions of \eqref{eq:weakthree}.

Now \eqref{heat_energy_basic} implies uniqueness as well as norm estimates for a solution $u = f + v$ of \eqref{eq:weaktwo}, using the triangle inequality and the support assumption for $f$. Hence, it remains to discuss existence of solutions.
This follows from a Galerkin approximation. 
To that end, we consider an eigenbasis $\{\varphi_k\}_{k=1}^{\infty}$ associated with the Dirichlet fractional Laplacian in $B_1$, i.e. 
\begin{align*}
(-\D)^s \varphi_k &= \lambda_k \varphi_k \mbox{ in } B_1,\\
\varphi_k & = 0 \mbox{ in } \R^n \setminus \overline{B}_1.
\end{align*}
We normalize these eigenfunctions so that they form an orthonormal basis of $\tilde{H}^s(B_1)$ and an orthogonal basis of $L^2(B_1)$. Thus, writing $\alpha_k(t) = (v(t), \varphi_k)_{L^2(B_1)}$, testing the equation \eqref{eq:weakthree} with $\varphi_k$, and requiring $\alpha_k(-1)=0$ results in the ODE 
\begin{align*}
\alpha_k'(t) + \lambda_k \alpha_k(t) &= \tilde{F}_{k}(t) \mbox{ for } t \in (-1,1),\\
\alpha_k(-1)& =0 ,
\end{align*}
where $\tilde{F}_{k}(t) := \tilde{F}(t)(\varphi_k)$. If $\alpha_k(t)$ solve these ODE, we define 
\[
v_N(x,t):= \sum\limits_{k=1}^{N} \alpha_k(t)\varphi_k(x).
\]
This function solves \eqref{eq:weakthree} with $\tilde{F}$ replaced by $\tilde{F}_N := \sum_{k=1}^N \tilde{F}_k(t) \varphi_k(x)$. Since $\{\varphi_k\}$ is an orthonormal basis of $\tilde{H}^s(B_1)$, functions of the form $w(t,x)=\sum_{k=1}^M w_k(t) \varphi_k(x)$ are dense in the space $L^2((-1,1), \tilde{H}^s(B_1))$. Consequently $\norm{\tilde{F}_N}_{L^2((-1,1), H^{-s}(B_1))} \leq \norm{\tilde{F}}_{L^2((-1,1), H^{-s}(B_1))}$ for each $N$, and the energy estimate \eqref{heat_energy_basic} applied to $v_N$ yields 
\[
\sup_{t \in (-1,1)} \norm{v_N(t)}_{L^2(B_1)} + \norm{v_N}_{L^2((-1,1), H^s(\R^n))} + \norm{\partial_t v_N}_{L^2((-1,1), H^{-s}(B_1))} \leq C.
\]
This yields enough compactness to extract a weak limit $v$ as $N\rightarrow \infty$. Testing the equation for $v$ with functions of the form $w(t,x)=\sum_{k=1}^M w_k(t) \varphi_k(x)$, which form a dense set, we obtain a solution $v$ to \eqref{eq:weakthree} satisfying the desired a priori bounds.
\end{proof}

For later reference, we also note the following spatial higher regularity result:

\begin{lem}
\label{lem:zero}
Let $u$ be a weak solution to \eqref{eq:weaktwo} or \eqref{eq:weaktwo_dual} with $f=0$ and $F \in L^2(B_1 \times (-1,1))$. Assume that $W \subset \R^n$ is a Lipschitz set with $\overline{W} \cap \overline{B}_1=\emptyset$. 
Then, for any $r\geq 0$
\begin{align*}
\|(-\D)^s u\|_{L^2((-1,1),H^r(W))}
\leq C \|F\|_{L^2(B_1 \times (-1,1))}.
\end{align*}
\end{lem}

\begin{proof}
By Lemma \ref{lem:well-posedness} we have $(-\D)^s u \in L^2((-1,1),H^{-s}(\R^n))$. Since for $x\in W $ the assumption that $u(x,t)=0$ implies that
\begin{align*}
(-\D)^s u (x,t) 
= \mathrm{p.v.}\ c_{s,n} \int\limits_{\R^n} \frac{u(x,t)-u(y,t)}{|x-y|^{n+2s}} \,dy 
=  - \mathrm{p.v.}\ c_{s,n} \int\limits_{\R^n} \frac{u(y,t)}{|x-y|^{n+2s}} \,dy,
\end{align*}
we have for $k \geq 0$ that $\norm{(-\D)^s u}_{L^2((-1,1), H^k(W))} \leq C_k \norm{u}_{L^2((-1,1) \times \R^n)}$. The claimed estimate follows from the $L^2((-1,1), H^{s}(\R^n))$ bound in Lemma \ref{lem:well-posedness}. 
\end{proof}

\subsection{Qualitative approximation}
We next approach the qualitative density properties of the fractional heat equation. By means of a duality argument as in \cite{GSU16} this is reduced to
unique continuation properties of the fractional Laplacian. 

\begin{thm}
\label{prop:approx_qual}
Let $s\in (0,1)$ and consider the operator $P_s $ from \eqref{eq:map}.
Assume that $W \subset \R^n$ is a Lipschitz set with $\overline{W}\cap \overline{B}_1 = \emptyset$.
Define 
\begin{align*}
\mathcal{R}:=\{u|_{B_1 \times (-1,1)}: u = P_s f, \ f \in C_c^{\infty}(W \times (-1,1))\}.
\end{align*}
Then the set $\mathcal{R}$ is dense in $L^2(B_1 \times (-1,1))$.
\end{thm}

\begin{rmk}
We emphasize that the choice of the spatial domain $B_1$ is not essential in our argument. It is for instance possible to consider more general, bounded Lipschitz domains.
\end{rmk}

\begin{proof}
By the Hahn-Banach theorem, it is enough to show that if $v \in L^2(B_1 \times (-1,1))$ satisfies 
\[
(P_s f, v)_{L^2(B_1 \times (-1,1))} = 0 \qquad \text{for all $f \in C^{\infty}_c(W \times (-1,1))$},
\]
then $v \equiv 0$. Now let $v$ be such a function. We consider the dual problem to \eqref{eq:eq_main}. It is given by
\begin{equation}
\label{eq:dual}
\begin{split}
(-\p_t +(-\D)^s) \varphi &=  v \mbox{ in } B_1 \times (-1,1),\\
 \varphi & = 0 \mbox{ in } (\R^n \setminus \overline{B}_1) \times (-1,1),\\
 \varphi & = 0 \mbox{ in } \R^n \times \{1\}.
\end{split}
\end{equation}
We note that by virtue of Lemma \ref{lem:well-posedness} and Remark \ref{rmk:weak_sol}, both \eqref{eq:eq_main} and \eqref{eq:dual} are well-posed.

Let now $f \in C^{\infty}_c(W \times (-1,1))$, let $u$ solve \eqref{eq:eq_main}, and let $\varphi$ solve \eqref{eq:dual}. Since $u-f$ vanishes outside $\overline{B}_1 \times (-1,1)$, it follows that 
\begin{equation} \label{ps_adjoint}
\begin{split}
(P_s f, v)_{L^2(B_1 \times (-1,1))} &= (u-f, (-\p_t + (-\D)^s)\varphi)_{L^2(\R^n \times (-1,1))}\\
&=  -( f, (-\D)^s \varphi)_{L^2(W \times (-1,1))}.
\end{split}
\end{equation}
In the last equality we used that $u$ is a solution, that $u(-1) = \varphi(1) = 0$, and the support conditions on $\varphi$ and $f$.

Since $(P_s f, v)_{L^2(B_1 \times (-1,1))} = 0$ for all $\varphi \in C^{\infty}_c(W \times (-1,1))$, the above computation yields that 
\begin{align*}
\varphi =
(-\D)^s \varphi = 0 \mbox{ in } W \times (-1,1). 
\end{align*}
By weak unique continuation for the fractional Laplacian (for each fixed time slice), see e.g.\ \cite[Theorem 1.2]{GSU16}, this implies that $\varphi(\cdot,t) = 0$ in $\R^n \times \{t\}$ for all $t\in (-1,1)$ and hence $v=0$. By the Hahn-Banach theorem this thus yields the desired density property.
\end{proof}

\begin{rmk}
The adjoint property \eqref{ps_adjoint} can also be inferred using the Caffarelli-Silvestre extension, see Section \ref{sec:quant}.
Denoting the Caffarelli-Silvestre extension associated with $\varphi(x,t)$ by $\overline{\varphi}(x,x_{n+1},t)$ and using the notation from \eqref{eq:Neumann}, the equation \eqref{eq:dual} can be formulated as 
\begin{equation}
\label{eq:harm_extend}
\begin{split}
(\p_{n+1} x_{n+1}^{1-2s}\p_{n+1} + x_{n+1}^{1-2s}\D')\overline{\varphi} &= 0 \mbox{ in } \R^{n+1}_+ \times (-1,1),\\
(\p_{n+1}^s-\p_t) \overline{\varphi} &= v \mbox{ in } B_1 \times \{0\}\times (-1,1),\\
\overline{\varphi} & = 0 \mbox{ in } (\R^n \setminus \overline{B}_1) \times \{0\}\times (-1,1),\\
\overline{\varphi} & = 0 \mbox{ in } \R^{n}\times \{0\}\times \{1\}.
\end{split}
\end{equation}
With this notation, we then have
\begin{equation}
\label{eq:HB}
\begin{split}
&(v,P_s f)_{L^2(B_1 \times (-1,1))}
= ((-\p_t + (-\D)^s)\varphi, P_s f)_{L^2(B_1 \times (-1,1))}\\
&= (\varphi, \p_t u)_{L^2(B_1 \times (-1,1))}
+ (\p_{n+1}^s\ov,  \ou)_{L^2(\R^n \times \{0\}\times (-1,1))} 
-  (\p_{n+1}^s\ov,  \ou)_{L^2(W \times \{0\}\times (-1,1))}\\
&= (\varphi, \p_t u)_{L^2(B_1 \times (-1,1))}
+ (\ov,  \p_{n+1}^s \ou)_{L^2(\R^n \times \{0\}\times (-1,1))} 
-  (\p_{n+1}^s\ov,  \ou)_{L^2(W \times \{0\}\times (-1,1))}\\
&= (\varphi, \p_t u)_{L^2(B_1 \times (-1,1))}
+ (\varphi,  (-\D)^s u)_{L^2(B_1\times (-1,1))} 
- ((-\D)^s \varphi,  f)_{L^2(W \times (-1,1))}\\
& = -((-\D)^s \varphi,  f)_{L^2(W \times (-1,1))}.
\end{split}
\end{equation}
Here we first integrated by parts in time, then used that $\ov$ and $\ou$ are solutions to the Caffarelli-Silvestre extension for each fixed time slice and finally exploited that $u$ obeys \eqref{eq:eq_main}. 
\end{rmk}

\begin{rmk}
\label{rmk:loc_nonloc}
The argument of Theorem \ref{prop:approx_qual} shows that also in the case, in which a \emph{local} operator is combined with a \emph{nonlocal} operator, the density properties of $\mathcal{R}$ are purely determined by the nonlocal component of the operator: The \emph{local} part of the operator does not play a role in the reduction to the weak unique continuation principle and only the weak unique continuation properties of the underlying \emph{nonlocal} operator are of relevance.

In particular, this implies that as in \cite{DSV16} the parabolic character of the problem at hand was not essential in the qualitative density argument. The same strategy can be pursued for more general operators, e.g.\ the fractional wave equation (c.f.\ Section \ref{sec:extend}).
\end{rmk}

In analogy to the notation from in control theory we use the following convention in the sequel:
 
\begin{defi}
\label{defi:control} 
Let $P_s$ for $s\in (0,1)$ be as in \eqref{eq:map}.
Given a function $h\in L^2(B_1 \times (-1,1))$ and an error threshold $\epsilon>0$, we refer to a function $f_{\epsilon,h}$, which satisfies
\begin{align*}
\|h-P_s f_{\epsilon,h}\|_{L^2(B_1 \times (-1,1))}< \epsilon,
\end{align*}
as a \emph{control function for $h$ with error threshold $\epsilon>0$}. If there is no danger of confusion, we also simply refer to it as a \emph{control}.
\end{defi}

\section{Propagation of smallness}
\label{sec:quant}

With the qualitative behaviour from the previous section at hand, we now proceed to \emph{quantitative} aspects of these approximation results. We begin our analysis by deducing a central propagation of smallness property, which quantifies the weak unique continuation result used in Section \ref{sec:qual} and  provides the basis for the proof of Theorem \ref{prop:cost}. This result is stated in terms of the Caffarelli-Silvestre extension (c.f. \cite{CS07}), which we now recall.

By virtue of \cite{CS07} it is possible to realize the nonlocal operator $(-\D)^s$ with $s\in(0,1)$ as a local operator by adding an additional dimension: Given a function $v\in L^2(\R^n)$, and writing $x = (x',x_{n+1}) \in \R^{n+1}$, we have that for some constant $c_s \in \R \setminus \{0\}$ 
\begin{align}
\label{eq:Neumann}
(-\D)^s v(x') = \p_{n+1}^s \overline{v}(x'):= c_s\lim\limits_{x_{n+1}\rightarrow 0} x_{n+1}^{1-2s}\p_{n+1}\overline{v}(x',x_{n+1}),
\end{align}
where the function $\bar{v}$ is a solution to
\begin{align*}
\nabla \cdot x_{n+1}^{1-2s} \nabla \overline{v} &= 0 \mbox{ in } \R^{n+1}_+,\\
\overline{v} &= v \mbox{ on } \R^n \times \{0\}.
\end{align*}
Here $\nabla=(\p_1,\dots, \p_{n+1})^t$ denotes the full gradient in $n+1$ dimensions (i.e.\ in the tangential and normal directions). If convenient, we also abbreviate the tangential part of it by $\nabla'$. In the sequel, we will use the convention that for a function $v \in L^2(\R^n)$ we denote its Caffarelli-Silvestre extension into $\R^{n+1}_+$ by $\overline{v}$.

\begin{prop}
\label{prop:small_s} 
Let $n \geq 1$, $s\in(0,1)$, and let $W \subset \R^n$ be a bounded Lipschitz domain with $\overline{B}_1 \cap \overline{W} = \emptyset$. There exist constants $C_j \geq 1$, $\mu_j \in (0,1)$, and $\sigma_j > 0$, only depending on $n$, $s$, and $W$, such that whenever $v\in L^2(B_1 \times (-1,1))$, $\varphi$ is the solution of \eqref{eq:dual} associated with $v$, and $\delta \in (0,1/2)$, one has 
\begin{align*}
\|\ov\|_{L^2(B_1  \times \{\delta\}\times (-1,1) )} &\leq C_1 \delta^{s-1} \norm{\partial_{n+1}^s \ov}_{L^2(W \times \{0\} \times (-1,1))}^{\mu_1 \delta^{\sigma_1}} \| v\|_{L^2(B_1 \times (-1,1))}^{1-\mu_1 \delta^{\sigma_1}}, \\
\|x_{n+1}^{1-2s} \partial_{n+1} \ov\|_{L^2(B_1  \times \{\delta\}\times (-1,1) )} &\leq C_2 \delta^{-s} \norm{\partial_{n+1}^s \ov}_{L^2(W \times \{0\} \times (-1,1))}^{\mu_2 \delta^{\sigma_2}} \| v\|_{L^2(B_1 \times (-1,1))}^{1-\mu_2 \delta^{\sigma_2}}.
\end{align*}
\end{prop}

We prove these estimates as a consequence of a combination of bulk and boundary three balls inequalities and a global estimate for solutions to \eqref{eq:dual}.

\begin{proof}
\emph{Step 1: Estimate for $\ov$.}
We first consider the estimate for $\ov$. We fix $t \in (-1,1)$ for the time being. Note that the function $\bar{\varphi}(\,\cdot\,,t)$ solves the degenerate elliptic equation 
\[
\nabla \cdot x_{n+1}^{1-2s}\nabla \bar{\varphi}(\,\cdot\,,t) = 0 \text{ in } \R^{n+1}_+, \qquad \bar{\varphi}(\,\cdot\,,t)|_{x_{n+1} = 0} = \varphi(\,\cdot\,,t).
\]
In particular, $\ov(\,\cdot\,,t)|_{(\R^n \setminus \overline{B}_1) \times \{0\}} = 0$. We wish to propagate the possible smallness of $\partial_{n+1}^s \ov(\,\cdot\,,t)$ on $W \times \{0\}$ to an estimate for $\ov(\,\cdot\,,t)$ on $B_1 \times \{\delta\}$.

We recall the following generalization of the three spheres inequality (c.f.\ \cite{ARV09} for a survey of these bounds in the case $s=1/2$) and of the Lebeau-Robbiano boundary-bulk interpolation estimate (c.f.\ \cite{LR95} for the case $s=1/2$) to solutions of the degenerate elliptic equation:
\begin{itemize}
\item[(i)] Setting
\begin{align*}
Q_{r,t}(x_0):= B_r(x_0')\times ((x_{0})_{n+1}+r, (x_{0})_{n+1} - r) \times \{t\},
\end{align*}
the following (weighted) three balls estimate holds (c.f.\ Propositions 5.3 and 5.4 in \cite{RS17})
\begin{align*}
\|x_{n+1}^{\frac{1-2s}{2}} \overline{\varphi}\|_{L^2(Q_{2r,t}^+(x_0))}
\leq C_{s} \|x_{n+1}^{\frac{1-2s}{2}}\overline{\varphi}\|_{L^2(Q_{r,t}^+(x_0))}^{\alpha}  \|x_{n+1}^{\frac{1-2s}{2}}\overline{\varphi}\|_{L^2(Q_{4r,t}^+(x_0))}^{1-\alpha}.
\end{align*}
Here $\alpha = \alpha(s) \in (0,1)$ and $C=C(s)>1$, $Q_{2r,t}^+:= Q_{2r,t} \cap \R^{n+1}_+$ and either $x_0 \in \R^n \times \{0\}$ and $B_{4r}'(x_0') \subset \R^n \setminus \overline{B}_1$, or $(x_0)_{n+1} \geq 5 r$.
\item[(ii)]
We have the following fractional bulk-boundary interpolation estimate due to Proposition 5.6 together with Remark 5.2 in \cite{RS17}, 
\begin{align*}
\|x_{n+1}^{\frac{1-2s}{2}}\ov\|_{L^2(W/2 \times [\ell/2, \ell] \times \{t\})}
\leq C \| \ov\|_{H^1(\R^{n+1}_+, \,x_{n+1}^{1-2s}\,dx)}^{1-\mu} \|\p_{n+1}^{s} \ov\|_{L^{2}(W \times \{0\}\times \{t\})}^{\mu},
\end{align*}
where $\ell \in (0,1]$, $\mu \in (0,1)$ and $C>1$ are constants depending on $n$, $s$ and $W$, and $W/2:=\{x\in W: \dist(x,\p W) > (\max_{z \in \overline{W}}\,\dist(z,\p W))/2\}$. We have also written 
\[
\| \ov\|_{H^1(\R^{n+1}_+, \,x_{n+1}^{1-2s}\,dx)}
 := \|x_{n+1}^{\frac{1-2s}{2}} \ov \|_{L^2(\R^{n+1}_+ \times \{t\})} + \|x_{n+1}^{\frac{1-2s}{2}} \nabla \ov \|_{L^2(\R^{n+1}_+ \times \{t\})}.
\]
\end{itemize}

Thus, using that $\ov$ is a Caffarelli-Silvestre extension of $\varphi$, for each fixed time $t\in(-1,1)$ and each radius $r$ with $0<r\leq (x_0)_{n+1}/5$, we can apply the three spheres inequality from (i) in the spatial variables $x=(x',x_{n+1})$ in the form
\begin{align}
\label{eq:3balla}
\frac{\|x_{n+1}^{\frac{1-2s}{2}}  \ov\|_{L^2(Q_{2r,t}(x_0))}}{ \|x_{n+1}^{\frac{1-2s}{2}} \ov\|_{L^2(\R^{n+1}_+)}} \leq 
C_1 \left[ \frac{\|x_{n+1}^{\frac{1-2s}{2}}  \ov\|_{L^2(Q_{r,t}(x_0))}}{ \|x_{n+1}^{\frac{1-2s}{2}} \ov\|_{L^2(\R^{n+1}_+)}} \right]^{\alpha}.
\end{align}
We consider a chain of $N$ balls, $K:=\bigcup\limits_{i=1}^{N} Q_{r_i,t}(x_i)$, which connects $W\times [\ell/2,\ell] \times \{t\}$ with $B_1 \times [\delta/2, 2\delta] \times \{t\} $ (see e.g.\ \cite[proof of Theorem 5.5]{RS17} for more details on this argument). Due to the constraint $(x_i)_{n+1} \geq 5 r_i$, we note that the constant $N$ can be chosen to be of the order 
\begin{align}
\label{eq:N}
N \sim C |\log(\delta)|, 
\end{align}
where $C>1$ is a constant that only depends on $n$, $s$, $W$ and may change from line to line.

Applying \eqref{eq:3balla} iteratively along this chain, we infer that
\begin{align}
\label{eq:3balls}
\|x_{n+1}^{\frac{1-2s}{2}} \ov\|_{L^2(B_1 \times [\delta/2,2\delta]\times \{t\})} \leq C_2
\|x_{n+1}^{\frac{1-2s}{2}} \ov\|_{L^2(W/2  \times [\ell/2, \ell]\times \{t\})}^{\alpha^N} \|x_{n+1}^{\frac{1-2s}{2}} \ov\|_{L^2(\R^{n+1}_+ \times \{t\})}^{1-\alpha^N},
\end{align}
where $\ell \in(0,1]$ is as in (ii), and $C_2 \leq C C_1^{1+\alpha + \ldots + \alpha^{N-1}} \leq C C_1^{\frac{1}{1-\alpha}}$ so $C_2$ is independent of $N$. 
By Caccioppoli's inequality \cite[Lemma 4.5]{RS17}, \eqref{eq:3balls} can be upgraded to read
\begin{align*}
& \|x_{n+1}^{\frac{1-2s}{2}} \ov\|_{L^2(B_1 \times [\delta/2,2\delta]\times \{t\})} + \delta \|x_{n+1}^{\frac{1-2s}{2}} \nabla \ov\|_{L^2(B_1 \times [3\delta/4,\delta]\times \{t\})}  \\
&\leq 
C \| x_{n+1}^{\frac{1-2s}{2}}\ov\|_{L^2(W/2 \times [\ell/2, \ell]\times \{t\})}^{\alpha^N} \|x_{n+1}^{\frac{1-2s}{2}} \ov\|_{L^2(\R^{n+1}_+ \times \{t\})}^{1-\alpha^N}.
\end{align*}
Combining this with a simple trace estimate (using the fundamental theorem of calculus) also yields
\begin{align*}\delta^{\frac{1-2s}{2}} \| \ov\|_{L^2(B_1 \times \{\delta\}\times \{t\})} 
 \leq C \delta^{-1/2} \| x_{n+1}^{\frac{1-2s}{2}}\ov\|_{L^2(W/2 \times [\ell/2, \ell]\times \{t\})}^{\alpha^N} \|x_{n+1}^{\frac{1-2s}{2}} \ov\|_{L^2(\R^{n+1}_+ \times \{t\})}^{1-\alpha^N}.
\end{align*}
Combining this with (ii), i.e.\ the analogue of the bulk-boundary interpolation estimate of Lebeau and Robbiano \cite{LR95}, further yields
\begin{equation}
\label{eq:3balls_d}
\| \ov\|_{L^2(B_1 \times \{\delta\}\times \{t\})} \leq 
C \delta^{s-1} \|  \p_{n+1}^s\ov\|_{L^2(W \times \{0\}\times \{t\} )}^{\mu\alpha^N} \| \ov\|_{H^1(\R^{n+1}_+, \,x_{n+1}^{1-2s}\,dx)}^{1-\mu\alpha^N}.
\end{equation}
Here we have used that $\ov=0$ on $W \times \{0\}\times \{t\}$.

Integrating the square of \eqref{eq:3balls_d} in time for $t\in(-1,1)$ and applying H\"older's inequality then gives
\begin{equation}
\label{eq:propagate}
\begin{split}
\| \ov\|_{L^2(B_1 \times \{\delta\}\times (-1,1))} &\leq 
C \delta^{s-1} \| \p_{n+1}^s\ov\|_{L^2(W  \times \{0\} \times (-1,1))}^{\mu\alpha^N} \\
 &\qquad \times \| \ov\|_{L^2((-1,1), H^1(\R^{n+1}_+, \,x_{n+1}^{1-2s} \,dx))}^{1-\mu\alpha^N}.
\end{split}
\end{equation}
By energy estimates for solutions to \eqref{eq:dual} (c.f. Lemma \ref{lem:well-posedness}) we further have
\begin{align}
\label{eq:global}
\|\varphi\|_{L^2((-1,1),H^{s}(\R^n))} \leq C \|v\|_{L^2(B_1 \times (-1,1))}.
\end{align}
Combining this with a boundary estimate for the Caffarelli-Silvestre extension, i.e.,
\begin{align*}
\|\overline{\varphi}\|_{L^2((-1,1),H^{1}(\R^{n+1}_+,\,x_{n+1}^{1-2s} \,dx))}
\leq C \|\varphi\|_{L^2((-1,1),H^{s}(\R^{n}))},
\end{align*}
and with equation \eqref{eq:propagate}, then allows us to conclude that
\begin{align}
\label{eq:propagate1}
\| \ov\|_{L^2(B_1 \times \{\delta\}\times (-1,1))} \leq 
C \delta^{s-1} \| \p_{n+1}^s \ov\|_{L^2(W\times \{0\} \times (-1,1) )}^{\mu\alpha^N} \| v\|_{L^2(B_1 \times (-1,1))}^{1-\mu\alpha^N}.
\end{align}
Recalling the bound from \eqref{eq:N} for $N$ therefore yields the claimed inequality for $\varphi$. \\

\emph{Step 2: Estimate for $x_{n+1}^{1-2s}\p_{n+1}\ov$.}
With the strategy from Step 1 at hand, we explain the necessary modifications for the estimate for $\psi(x):=x_{n+1}^{1-2s}\p_{n+1}\ov(x)$. To this end we use duality, which gives that if $\ov$ is a solution to 
\begin{align*}
\nabla \cdot x_{n+1}^{1-2s} \nabla \ov = 0 \mbox{ in } \R^{n+1}_+, \
\lim\limits_{x_{n+1}\rightarrow 0} x_{n+1}^{1-2s} \p_{n+1} \ov
= g \mbox{ on } \R^n \times \{0\},
\end{align*}
then $\psi(x):= x_{n+1}^{1-2s}\p_{n+1}\ov$ is a solution to  
\begin{align*}
\nabla \cdot x_{n+1}^{1-2\tilde{s}} \nabla \psi = 0 \mbox{ in } \R^{n+1}_+, \ 
 \psi = g \mbox{ on } \R^n \times \{0\},
\end{align*}
with $\tilde{s}=1-s$ (c.f. \cite{CS07} and \cite{CS14}). Thus, in the interior of the upper half-plane we can argue analogously as in Step 1 and infer that with the notation of Step 1 
\begin{align*}
\| \psi \|_{L^2(B_1 \times \{\delta\}\times \{t\})} 
 \leq C \delta^{\tilde{s}-1} \| x_{n+1}^{\frac{1-2\tilde{s}}{2}} \psi \|_{L^2(W/2 \times [3\ell/4, 7\ell/8]\times \{t\})}^{\alpha^N} \|x_{n+1}^{\frac{1-2\tilde{s}}{2}} \psi \|_{L^2(\R^{n+1}_+ \times \{t\})}^{1-\alpha^N}.
\end{align*}
Spelling out the definition of $\psi$ then yields
\begin{align*}
\| x_{n+1}^{1-2s}\p_{n+1}\ov \|_{L^2(B_1 \times \{\delta\}\times \{t\})} \leq 
C \delta^{-s}  \| x_{n+1}^{\frac{1-2s}{2}}\p_{n+1}\ov \|_{L^2(W \times [3\ell/4, 7\ell/8]\times \{t\})}^{\alpha^N} \|x_{n+1}^{\frac{1-2s}{2}} \p_{n+1}\ov \|_{L^2(\R^{n+1}_+ \times \{t\})}^{1-\alpha^N}.
\end{align*}
Invoking Caccioppoli's inequality thus entails
\begin{align*}
\| x_{n+1}^{1-2s}\p_{n+1}\ov \|_{L^2(B_1 \times \{\delta\}\times \{t\})} \leq 
C \delta^{-s}   \| x_{n+1}^{\frac{1-2s}{2}} \ov \|_{L^2(W \times [\ell/2, \ell]\times \{t\})}^{\alpha^N} \|x_{n+1}^{\frac{1-2s}{2}} \p_{n+1}\ov \|_{L^2(\R^{n+1}_+ \times \{t\})}^{1-\alpha^N}.
\end{align*}
This, however, is in a form which allows us to apply the bulk-boundary interpolation estimate from (ii), whence
\begin{align}
\label{eq:interior_1}
\| x_{n+1}^{1-2s}\p_{n+1}\ov \|_{L^2(B_1 \times \{\delta\}\times \{t\})} \leq 
C \delta^{-s}  \|\p_{n+1}^s \ov \|_{L^2(W \times \{0\}\times \{t\})}^{\mu \alpha^N} \| \ov\|_{H^1(\R^{n+1}_+, \,x_{n+1}^{1-2s}\,dx)}^{1-\mu \alpha^N}.
\end{align}
Combining this with the energy estimate from \eqref{eq:global} therefore leads to the desired estimate for $x_{n+1}^{1-2s}\p_{n+1}\ov$.
\end{proof}

\begin{rmk}
\label{rmk:quant}
The argument for Proposition \ref{prop:small_s} can be regarded as consisting of two main ingredients: On the one hand, we exploit (interior and boundary) three balls arguments and propagation of smallness properties for solutions to \eqref{eq:harm_extend}. This leads to the bound in \eqref{eq:propagate} and only depends on the underlying \emph{nonlocal} operator (and its localization by means of the harmonic extension).
On the other hand, we combine these propagation of smallness results with a global energy estimate, c.f.\ \eqref{eq:global}.   
It is only at this point, at which we have made use of the full equation with its \emph{local and nonlocal} contributions, i.e.\ only at this point the parabolic nature of the problem is exploited. 
\end{rmk}

\section{Proof of Theorem \ref{prop:cost}} \label{sec:proof_main}

With the quantitative uniqueness result from Proposition \ref{prop:small_s} at hand, we now proceed to \emph{quantitative approximation} results. Here we are interested in estimating the \emph{cost} of approximation: More precisely, for a given function $h\in L^2(B_1 \times (-1,1))$ and an error threshold $\epsilon>0$, we seek to derive bounds on the size of suitable norms of a possible control function $f_{\epsilon,h}$ (in dependence of suitable norms of $h$ and of $\epsilon>0$). This will prove the main approximation result of Theorem \ref{prop:cost}.

We follow the variational approach presented in \cite{FZ00}. We thus characterise $f_{\epsilon,h}$ in terms of the minimizer of the functional
\begin{align}
\label{eq:functional}
\Je(v) = \frac{1}{2}\int\limits_{W \times (-1,1)}|\eta(-\D)^{s} \varphi|^2 \,dx \,dt + \epsilon \|v\|_{L^2(B_1 \times (-1,1))} - \int\limits_{B_1 \times (-1,1)} h v \,dx \,dt.
\end{align}
Here $\varphi$ and $v$ are related through \eqref{eq:dual}, and $\eta \in C^{\infty}_c(W)$ is a cutoff function satisfying $0 \leq \eta \leq 1$ and $\eta = 1$ on $W/2:=\{x\in W: \dist(x,\p W) > (\max_{z \in \overline{W}}\,\dist(z,\p W))/2\}$. If $0 < s < 1/2$ we could replace $\eta$ by the characteristic function $\chi_{W}$, but if $s \geq 1/2$ then $\chi_W$ is not a pointwise multiplier on $H^s(\mR^n)$ and we need to use a smooth cutoff.

In order to prove the result of Theorem \ref{prop:cost}, we argue in three steps, which we split into three lemmata: We first show that, for a given function $h$ and an error threshold $\epsilon>0$, a unique minimizer $\hat{v}$ of the functional \eqref{eq:functional} exists (Lemma \ref{lem:proof1}). This is a consequence of the weak unique continuation properties of the fractional Laplacian. Secondly, if $\hat{\varphi}$ is the solution of \eqref{eq:dual} corresponding to $\hat{v}$, we argue that $f:=-\eta^2 (-\D)^{s}\hat{\varphi}$ is a control for $h$ corresponding to an error threshold $\epsilon>0$ (i.e., that it satisfies the first estimate in \eqref{eq:approx_cost}). This follows from minimality (Lemma \ref{lem:proof2}). Finally, in the last step (Lemma \ref{lem:proof3}), we provide the bound on the cost of approximation (i.e., the second estimate in \eqref{eq:approx_cost}). This relies on the estimates from Proposition \ref{prop:small_s}.

\begin{lem}[Existence of minimizers]
\label{lem:proof1}
Let $s\in(0,1)$, $n\geq 1$, $\epsilon >0$. Assume that $h \in H^1_0(B_1 \times (-1,1))$. Let 
\begin{align*}
\Je: L^2(B_1 \times (-1,1))\rightarrow \R \cup \{\pm \infty\}, \
v \mapsto \Je(v)
\end{align*}
be as in \eqref{eq:functional}. Then there exists a unique minimizer $\hat{v}\in L^2(B_1\times (-1,1))$ of $\Je$.
\end{lem}

\begin{proof}
It is enough to prove that $\Je$ is strictly convex, continuous, and coercive, since then it will have a unique minimizer (see e.g.\ \cite[Section II.1]{EkelandTemam}). The functional $\Je$ is convex since it is the sum of three convex functionals, and it is strictly convex since $v \mapsto \| \eta (-\Delta)^s \varphi \|_{L^2(W \times (-1,1))}^2$ is strictly convex (this uses again weak unique continuation for the fractional Laplacian). In addition, $\Je$ is continuous since it is the sum of three continuous functionals: The fact that $v \mapsto \| \eta (-\Delta)^s \varphi \|_{L^2(W \times (-1,1))}^2$ is continuous follows since $(-\D)^{s}\varphi$ is evaluated at $W \times (-1,1)$, where $\varphi=0$ and where according to Lemma \ref{lem:zero} strong elliptic regularization is present.

Hence, it suffices to prove coercivity of \eqref{eq:functional} to obtain the existence of minimizers. This will be reduced to the weak unique continuation property for the fractional Laplacian. To this end, let $v_k \in L^2(B_1 \times (-1,1))$ be a sequence such that $\|v_k\|_{L^2(B_1 \times (-1,1))} \rightarrow \infty$. We seek to show that
\begin{align*}
\Je(v_k) \rightarrow \infty \mbox{ as } k \rightarrow \infty.
\end{align*}
Abbreviating the corresponding normalized functions by $\hat{v}_k:= \frac{v_k}{\|v_k\|_{L^2(B_1 \times (-1,1))}}$ and the associated solutions to \eqref{eq:dual} by $\hat{\varphi}_k$, we have that
\begin{align*}
\frac{\Je(v_k)}{\|v_k\|_{L^2(B_1 \times (-1,1))}}
= \frac{\|v_k\|_{L^2(B_1 \times (-1,1))}}{2}\|\eta(-\D)^{s}\hat{\varphi}_k\|_{L^2(W \times (-1,1))}^2 + \epsilon  - \int\limits_{B_1 \times (-1,1)} h \hat{v}_k \,dx \,dt.
\end{align*} 

We now distinguish two scenarios: If on the one hand $\liminf\limits_{k \rightarrow \infty} \|\eta(-\D)^{s}\hat{\varphi}_k\|_{L^2(W \times (-1,1))}>0$, then the normalization of $\hat{v}_k$ and the divergence of $\|v_k\|_{L^2(B_1 \times (-1,1))}$ imply that 
\begin{align*}
\liminf\limits_{k \rightarrow \infty} \frac{\Je(v_k)}{\|v_k\|_{L^2(B_1 \times (-1,1))}}  \rightarrow \infty, 
\end{align*}
which proves the desired coercivity. 

If on the other hand, $\lim\limits_{k \rightarrow \infty} \|\eta(-\D)^{s}\hat{\varphi}_k\|_{L^2(W \times (-1,1))} =0$ (here and below we understand that we have passed to a suitable subsequence), we deduce coercivity from the weak unique continuation property of the limiting problem as $k\rightarrow \infty$. More precisely, we note that:
\begin{itemize}

\item
By virtue of the normalization and the Banach-Alaoglu theorem,
$\hat{v}_k \rightharpoonup \hat{v} \mbox{ in } L^2(B_1 \times (-1,1))$ for some $\hat{v} \in L^2(B_1 \times (-1,1))$.

\item
Energy estimates and the weak form of the equation imply that 
\[
\|\hat{\varphi}_k\|_{L^2((-1,1),H^{s}(\R^n))} \leq C < \infty
\]
uniformly in $k\in \N$ (c.f. \eqref{eq:global} and Lemma \ref{lem:well-posedness}). Thus, $\hat{\varphi}_k \rightharpoonup \hat{\varphi} $ in $L^2(\R^n \times (-1,1))\cap L^2((-1,1), \tilde{H}^s(B_1))$ for some $\hat{\varphi} \in L^2(\R^n \times (-1,1))\cap L^2((-1,1), \tilde{H}^s(B_1))$ with $\hat{\varphi}|_{(\R^n \setminus \overline{B}_1) \times (-1,1)} = 0$. This also implies that $(-\partial_t + (-\Delta)^s) \hat{\varphi}_k \to (-\partial_t + (-\Delta)^s) \hat{\varphi}$ in $\mathcal{D}'(\R^n \times (-1,1))$. Since $\hat{\varphi}_k$ is a solution corresponding to $\hat{v}_k$ we obtain that 
\[
(-\partial_t + (-\Delta)^s) \hat{\varphi} = \hat{v} \text{ in $B_1 \times (-1,1)$.}
\]

\item
Since further $\eta(-\D)^{s}\hat{\varphi}_k  \rightarrow 0$ in $L^2(W\times (-1,1))$, this discussion shows that
\begin{align*}
\hat{\varphi}=0 \mbox{ and } (-\D)^{s}\hat{\varphi} =0 \mbox{ in } W/2 \times (-1,1).
\end{align*}
\end{itemize}
As $(-\D)^s \hat{\varphi} \in L^2((-1,1), H^{-s}(\R^n))$, (spatial) weak unique continuation applied at a.e. time slice (see e.g.\ \cite[Theorem 1.2]{GSU16}) however implies that $\hat{\varphi}=0$ for a.e. $t\in(-1,1)$ and thus $\hat{v}=0$. As a consequence, $\int\limits_{B_1 \times (-1,1)} h \hat{v}_k \,dx \,dt \rightarrow 0$, so that for a sufficiently large choice of $k\in \N$
\begin{align*}
\frac{\Je(v_k)}{\|v_k\|_{L^2(B_1 \times (-1,1))}}
\geq \frac{\|v_k\|_{L^2(B_1 \times (-1,1))}}{2}\|\eta(-\D)^{s}\hat{\varphi}_k\|_{L^2(W \times (-1,1))}^2 + \frac{\epsilon}{2} \geq \frac{\epsilon}{2},
\end{align*}
which also implies the claimed coercivity.
\end{proof}

With existence of a minimizer at hand, we address the approximation property:

\begin{lem}[Approximation]
\label{lem:proof2}
Let $s\in(0,1)$, $n\geq 1$, $\epsilon >0$. Assume that $h \in H^1_0(B_1 \times (-1,1))$. Let $\Je$
be the functional from \eqref{eq:functional} and let $\hat{v}$ be its unique minimizer. Denote by $\hat{\varphi}$ the solution to \eqref{eq:dual} with inhomogeneity $\hat{v}$, and let $f := - \eta^2 (-\D)^s \hat{\varphi}$.
Then the solution $u$ of \eqref{eq:eq_main} satisfies
\begin{align}
\label{eq:approx}
\|u-h\|_{L^2(B_1 \times (-1,1))} \leq \epsilon.
\end{align}
Moreover, $f \in L^2((-1,1), C^{\infty}_c(W))$ and 
\begin{align*}
\|f\|_{L^2(W \times (-1,1))}^2 \leq -2 \Je(\hat{v}).
\end{align*}
\end{lem}

\begin{proof}
Let $\hat{v}$ be the minimizer of the problem \eqref{eq:functional} and let $\hat{\varphi}$ be the corresponding solution of \eqref{eq:dual}.

The approximation property in \eqref{eq:approx} then follows from spelling out the minimality condition
\begin{align*}
\Je(\hat{v}+\mu v)-\Je(\hat{v}) \geq 0,
\end{align*}
for all $\mu \in \R$, combined with the triangle inequality to estimate the difference of the $L^2$ norms and by passing to the limit $\mu \rightarrow 0_{\pm}$. Indeed,
\begin{align}
0 &\leq \frac{1}{2}\|\eta(-\D)^{s}(\hat{\varphi} + \mu \varphi)\|_{L^2(W \times (-1,1))}^2 - \frac{1}{2}\|\eta(-\D)^{s}\hat{\varphi} \|_{L^2(W \times (-1,1))}^2 \notag \\
& \quad + \epsilon \|\hat{v} + \mu v\|_{L^2(B_1 \times (-1,1))}
- \epsilon \|\hat{v}\|_{L^2(B_1 \times (-1,1))}
- \mu \int\limits_{B_1 \times (-1,1)} h v \,dx \,dt \notag \\
& \leq \mu \int\limits_{W \times (-1,1)} \eta^2 (-\D)^{s} \hat{\varphi} (-\D)^{s} \varphi \,dx \,dt + \frac{\mu^2}{2} \int\limits_{W \times (-1,1)} \eta^2|(-\D)^{s} \varphi|^2 \,dx \,dt \notag \\
 & \quad + \epsilon |\mu| \|v\|_{L^2(B_1 \times (-1,1))} - \mu \int\limits_{B_1 \times (-1,1)} h v \,dx \,dt. \label{minimization_argument}
\end{align}
Dividing by $\mu\neq 0$ and passing to the limits $\mu\rightarrow 0_{\pm}$, we obtain
\begin{align}
\label{eq:EL}
\left|\int\limits_{W \times (-1,1)} \eta^2 (-\D)^{s} \hat{\varphi} (-\D)^{s}\varphi \,dx \,dt - \int\limits_{B_1 \times (-1,1)} h v \,dx \,dt \right| \leq \epsilon \|v\|_{L^2(B_1 \times (-1,1))}.
\end{align}
Here $\varphi$ denotes the solution to \eqref{eq:dual} corresponding to $v\in L^2(B_1 \times (-1,1))$.
Defining $f:= - \eta^2 (-\D)^{s} \hat{\varphi}$ and denoting the associated solution to \eqref{eq:eq_main} by $u$, an analogous computation as in \eqref{eq:HB} implies that \eqref{eq:EL} turns into 
\begin{align}
\label{eq:EL1}
\left|\int\limits_{B_1 \times (-1,1)} (u-h)v  \,dx \,dt\right| \leq \epsilon \|v\|_{L^2(B_1 \times (-1,1))}.
\end{align}
By duality this yields \eqref{eq:approx}. One also has $f \in L^2((-1,1), C^{\infty}_c(W))$ by Lemma \ref{lem:zero}.

We note that choosing $v = \hat{v}$ and repeating the argument leading to \eqref{minimization_argument} (where one now avoids the triangle inequality) gives for $|\mu|$ small 
\[
0 \leq \mu \| \eta (-\D)^{s} \hat{\varphi} \|_{L^2}^2 + \frac{\mu^2}{2} \| \eta (-\D)^{s} \hat{\varphi} \|_{L^2}^2 + \epsilon \mu \|\hat{v}\|_{L^2} - \mu \int\limits_{B_1 \times (-1,1)} h \hat{v} \,dx \,dt.
\]
Dividing by $\mu \neq 0$ and letting $\mu \to 0_{\pm}$ implies that 
\begin{align*}
\int\limits_{W \times (-1,1)} |\eta (-\D)^{s} \hat{\varphi}|^2  \,dx \,dt + \epsilon \|\hat{v}\|_{L^2(B_1 \times (-1,1))} - \int\limits_{B_1 \times (-1,1)} h \hat{v} \,dx \,dt = 0,
\end{align*}
which directly leads to
\begin{align*}
\Je(\hat{v}) = - \frac{1}{2}\|\eta (-\D)^{s}\hat{\varphi}\|_{L^2(W \times (-1,1))}^2.
\end{align*}
Finally, since $0 \leq \eta \leq 1$ we have 
\[
\norm{f}_{L^2(W \times (-1,1))}^2 = \int\limits_{W \times (-1,1)} \eta^4 \abs{(-\Delta)^s \hat{\varphi}}^2 \,dx \,dt \leq \norm{\eta (-\Delta)^s  \hat{\varphi}}_{L^2(W \times (-1,1))}^2 = -2 \Je(\hat{v}).
\]
\end{proof}

Last but not least, we estimate the cost of control.

\begin{lem}[Cost of control]
\label{lem:proof3}
Let $s\in(0,1)$, $n\geq 1$, $\epsilon >0$. Assume that $h \in H^1_0(B_1 \times (-1,1))$. Let $\Je$
be the functional from \eqref{eq:functional} and let $\hat{v}$ be its unique minimizer. Denote by $\hat{\varphi}$ the solution to \eqref{eq:dual} with inhomogeneity $\hat{v}$.
Then we have that $f := - \eta^2 (-\D)^s \hat{\varphi}$ satisfies, for some $C$ and $\sigma$ only depending on $n, s, W$, 
\begin{align}
\label{eq:approx_a}
\|f\|_{L^2(W \times (-1,1))}\leq C e^{C (1+\|h\|_{H^1(B_1 \times (-1,1))}^{\sigma})\epsilon^{-\sigma}}\|h\|_{H^1(B_1 \times (-1,1))}.
\end{align}
\end{lem}

\begin{proof}
In order to finally provide the estimate on the cost of control, we consider a second functional in addition to $\Je(v)$:
\begin{align*}
\Jes(v) & := \frac{1}{2}\int\limits_{W \times (-1,1)}|\eta(-\D)^{s} \varphi|^2 \,dx \,dt + \frac{\epsilon}{2} \|v\|_{L^2(B_1 \times (-1,1))} \\
& \quad - \int\limits_{B_1 \times (-1,1)} h [(-\p_t + \p_{n+1}^s)\overline{\varphi}](x,\delta,t) \,dx \,dt,
\end{align*}
where, with slight abuse of notation, we write $\p_{n+1}^s \ov (x,\delta,t):=c_s \delta^{1-2s}\p_{n+1}\ov|_{(x,\delta,t)}$ and $\p_{n+1}^s\ov(x,0,t)=c_s \p_{n+1}^s \ov(x,t)$ (in the sense of Section \ref{sec:qual}).
As in \cite{FZ00} we rewrite our original functional from \eqref{eq:functional} as
\begin{align*}
\Je(v) &= \Jes(v) + \frac{\epsilon}{2} \|v\|_{L^2(B_1 \times (-1,1))} \\
 &\qquad + \int\limits_{B_1 \times (-1,1)} h (-\p_t + \p_{n+1}^s)[\overline{\varphi}(x,\delta,t)-\overline{\varphi}(x,0,t)] \,dx \,dt.
\end{align*}
Here we used that $\ov(x,0,t) = \varphi(x,t)$ and that $\varphi$ solves \eqref{eq:dual}. If we can ensure that
\begin{align}
\label{eq:pos}
\frac{\epsilon}{2} \|v\|_{L^2(B_1 \times (-1,1))} + \int\limits_{B_1 \times (-1,1)} h (-\p_t + \p_{n+1}^s)[\overline{\varphi}(x,\delta,t)-\overline{\varphi}(x,0,t)] \,dx \,dt \geq 0,
\end{align}
we then obtain that
\begin{align*}
I_1:= \min\limits_{v\in L^2(B_1 \times (-1,1))}\Je(v)\geq I_2 := \inf\limits_{v\in L^2(B_1 \times (-1,1))}\Jes(v).
\end{align*}
Since by Lemma \ref{lem:proof2}, $\|f\|_{L^2(W \times (-1,1))}^2 \leq -2 \min\,\Je(v) = -2 I_1$, this translates into
\begin{align*}
\|f\|_{L^2(W \times (-1,1))}^2\leq -2I_2 .
\end{align*}
It thus remains to estimate $I_2$ and to ensure \eqref{eq:pos}. We split the argument for this into two steps.\\

\emph{Step 1: Estimate for $I_2$.}
This follows from Proposition \ref{prop:small_s} (applied with $W/2$)
and the assumption that $h \in H^1_0(B_1 \times (-1,1))$. Indeed, if $\delta \in (0,1/2)$, we have
\begin{align*}
&\int\limits_{B_1 \times (-1,1)} h(x,t) [(-\p_t + \p_{n+1}^s)\overline{\varphi}](x,\delta,t) \,dx \,dt \\
&= \int\limits_{B_1 \times (-1,1)} (\p_t h)(x,t) \ov(x,\delta,t) \,dx \,dt + 
\int\limits_{B_1 \times (-1,1)}  h(x,t) \p_{n+1}^s \ov(x,\delta,t) \,dx \,dt \\
&\leq \|h\|_{H^1(B_1 \times (-1,1))} (\|\ov\|_{L^2(B_1  \times \{\delta\}\times (-1,1) )} + \|\partial_{n+1}^s \ov\|_{L^2(B_1  \times \{\delta\}\times (-1,1) )}) \\
&\leq \|h\|_{H^1(B_1 \times (-1,1))} (C_1 \delta^{s-1} \norm{\partial_{n+1}^s \ov}_{L^2(W/2 \times \{0\} \times (-1,1))}^{\mu_1 \delta^{\sigma_1}} \| v\|_{L^2(B_1 \times (-1,1))}^{1-\mu_1 \delta^{\sigma_1}}\\
 &\hspace{100pt} + C_2 \delta^{-s} \norm{\partial_{n+1}^s \ov}_{L^2(W/2 \times \{0\} \times (-1,1))}^{\mu_2 \delta^{\sigma_2}} \| v\|_{L^2(B_1 \times (-1,1))}^{1-\mu_2 \delta^{\sigma_2}}).
\end{align*}
Applying Young's inequality in the form $ab \leq (p'r)^{-p/p'} a^p/p + r b^{p'}$, and choosing $r = \tilde{\epsilon}/4$ where $\tilde{\epsilon} = \|h\|_{H^1(B_1 \times (-1,1))}^{-1} \epsilon$, implies that the first term on the right satisfies 
\begin{align*}
 &C_1 \delta^{s-1} \norm{\partial_{n+1}^s \ov}_{L^2(W/2 \times \{0\} \times (-1,1))}^{\mu_1 \delta^{\sigma_1}} \| v\|_{L^2(B_1 \times (-1,1))}^{1-\mu_1 \delta^{\sigma_1}} \\
 &\leq \left( \frac{\tilde{\epsilon}}{4(1-\mu_1 \delta^{\sigma_1})} \right)^{-\frac{1-\mu_1 \delta^{\sigma_1}}{\mu_1 \delta^{\sigma_1}}} (\mu_1 \delta^{\sigma_1}) (C_1 \delta^{s-1})^{1/(\mu_1 \delta^{\sigma_1})} \norm{\partial_{n+1}^s \ov}_{L^2(W/2 \times \{0\} \times (-1,1))} \\
 &\hspace{30pt} + \frac{\tilde{\epsilon}}{4} \| v\|_{L^2(B_1 \times (-1,1))} \\ 
 &\leq e^{C(1+|\log(\tilde{\epsilon})|)/\delta^{\sigma}} \norm{\partial_{n+1}^s \ov}_{L^2(W/2 \times \{0\} \times (-1,1))} + \frac{\tilde{\epsilon}}{4} \| v\|_{L^2(B_1 \times (-1,1))}.
\end{align*}
Arguing similarly for the second term, and recalling the definition of $\tilde{\epsilon}$, we obtain 
\begin{align*}
&\int\limits_{B_1 \times (-1,1)} h(x,t) [(-\p_t + \p_{n+1}^s)\overline{\varphi}](x,\delta,t) \,dx \,dt \\
&\quad \leq \|h\|_{H^1(B_1 \times (-1,1))} e^{C(1+|\log(\epsilon \|h\|_{H^{1}(B_1 \times (-1,1))}^{-1})|)/\delta^{\sigma}} \|\p_{n+1}^s\ov\|_{L^2(W/2 \times \{0\} \times (-1,1))}\\
& \quad \quad + \frac{\epsilon}{2}\|v\|_{L^2(B_1 \times (-1,1))} .
\end{align*}
Therefore, Young's inequality, and the fact that $\eta=1$ on $W/2$, yield 
\begin{align}
\label{eq:est1}
\begin{split}
I_2 
&\geq \inf\limits_{v\in L^2(B_1 \times (-1,1))}\left(
\frac{1}{2}\|\eta(-\D)^s \varphi \|_{L^2(W\times (-1,1))}^2 + \frac{\epsilon}{2}\|v\|_{L^2(B_1 \times (-1,1))} \right.\\
& \quad
- \|h\|_{H^1(B_1 \times (-1,1))} e^{C(1+|\log(\epsilon \|h\|_{H^{1}(B_1 \times (-1,1))}^{-1})|)/\delta^{\sigma}} \|(-\D)^s \varphi\|_{L^2(W/2 \times \{0\} \times (-1,1))}\\
& \quad \left. - \frac{\epsilon}{2}\|v\|_{L^2(B_1 \times (-1,1))} \right) \\
& \geq -  e^{C(1+|\log(\epsilon \|h\|_{H^1(B_1 \times (-1,1))}^{-1})|)/\delta^{\sigma}}\|h\|_{H^1(B_1 \times (-1,1))}^2.
\end{split}
\end{align}

\emph{Step 2: Ensuring \eqref{eq:pos}.}
In order to conclude the proof of Theorem \ref{prop:cost}, it suffices to ensure that \eqref{eq:pos} is satisfied and to deduce from this the resulting requirements on $\epsilon$ and $\delta$. To this end, we observe that
\begin{equation}
\label{eq:rel_d_e}
\begin{split}
&I:=\int\limits_{B_1 \times (-1,1)} h (-\p_t + \p_{n+1}^s)[\overline{\varphi}(x,\delta,t)-\overline{\varphi}(x,0,t)] \,dx \,dt \\
& = \int\limits_{B_1 \times (-1,1)} (\p_t h) [\overline{\varphi}(x,\delta,t)-\overline{\varphi}(x,0,t)] \,dx \,dt
+ \int\limits_{B_1 \times (-1,1)} h \p_{n+1}^s[\overline{\varphi}(x,\delta,t)-\overline{\varphi}(x,0,t)] \,dx \,dt,
\end{split}
\end{equation}
where we integrated by parts. 
We discuss these contributions separately in the sequel.\\
On the one hand, the fundamental theorem of calculus yields
\begin{align*}
&\left| \int\limits_{B_1 \times (-1,1)} \p_t h (\ov(x,\delta,t)-\ov(x,0,t)) \,dx \,dt \right|
= \left| \int\limits_{B_1 \times (-1,1)} \p_t h \int\limits_{0}^{\delta}\p_{z} \ov(x,z,t) \,dz \,dx \,dt \right|\\
& = \left| \int\limits_{B_1 \times (-1,1)} \p_t h \int\limits_{0}^{\delta}z^{\frac{2s-1}{2}}z^{\frac{1-2s}{2}}\p_{z} \ov(x,z,t) \,dz \,dx \,dt \right| \\
& \leq (2s)^{-1/2} \delta^{s} \| \p_t h \|_{L^2(B_1 \times (-1,1))} \|x_{n+1}^{\frac{1-2s}{2}}\p_{n+1} \ov \|_{L^2(\R^{n+1}_+\times(-1,1))}\\
& \leq C \delta^{s} \| \p_t h \|_{L^2(B_1 \times (-1,1))} \|v \|_{L^2(B_1 \times(-1,1))}.
\end{align*}
In the last line we here used the energy estimate \eqref{eq:global} to infer the bound
\begin{align*}
 \|x_{n+1}^{\frac{1-2s}{2}}\p_{n+1} \ov \|_{L^2(\R^{n+1}_+\times(-1,1))} \leq C \|v\|_{L^2(\R^n)}.
\end{align*}
On the other hand,
\begin{align*}
&\left| \int\limits_{B_1 \times (-1,1)} h  \p_{n+1}^s[\overline{\varphi}(x,\delta,t)-\overline{\varphi}(x,0,t)] \,dx \,dt \right| \\
& = \left| \int\limits_{B_1 \times (-1,1)} h \int\limits_{0}^{\delta} 
\p_z (z^{1-2s} \p_z \overline{\varphi}) \,dz \,dx \,dt \right|
 = \left| \int\limits_{B_1 \times (-1,1)} h \int\limits_{0}^{\delta} 
z^{1-2s} \D' \overline{\varphi} \,dz \,dx \,dt \right| \\ 
& = \left| \int\limits_{B_1 \times (-1,1)} \nabla' h \cdot \int\limits_{0}^{\delta} 
z^{\frac{1-2s}{2}} z^{\frac{1-2s}{2}} \nabla' \overline{\varphi} \,dz \,dx \,dt \right| \\ 
& \leq \|\nabla' h\|_{L^2(B_1 \times (-1,1))} 
(2-2s)^{-1/2} \delta^{1-s} \norm{z^{\frac{1-2s}{2}} \nabla' \ov}_{L^2(\R^{n+1}_+ \times (-1,1))} \\
& \leq C \delta^{1-s} \|\nabla' h\|_{L^2(B_1 \times (-1,1))} \|v \|_{L^2(B_1 \times(-1,1))}.
\end{align*}
Thus, inserting this into \eqref{eq:pos}, we obtain the following condition on $\delta, \epsilon$:
\begin{align*}
0<\delta \leq \left( \frac{\epsilon}{C\|h\|_{H^1(B_1 \times (-1,1))} + 1} \right)^{\frac{1}{\max\{s,1-s\}}}.
\end{align*}
Defining $\delta$ as saturating the upper bound in this estimate and plugging it into \eqref{eq:est1} then finally results in
\begin{align*}
\|f\|_{L^2(W \times (-1,1))} \leq C \exp\left( \frac{1+\|h\|_{H^1(B_1 \times (-1,1))}}{\epsilon} \right)^{\sigma}\|h\|_{H^1(B_1 \times (-1,1))},
\end{align*}
where $C > 1$ and $\sigma>0$ depend on $n$, $s$, and $W$.
\end{proof}

As an immediate consequence of Lemmas \ref{lem:proof1}-\ref{lem:proof3} we infer the result of Theorem \ref{prop:cost}.

\begin{proof}[Proof of Theorem \ref{prop:cost}]
Theorem \ref{prop:cost} follows by combining Lemmas \ref{lem:proof1}-\ref{lem:proof3}.
\end{proof}

\section{Extensions to More General Operators}
\label{sec:extend}

The arguments presented in Sections \ref{sec:qual}-\ref{sec:quant} extend to a much more general class of operators. In the sequel, we briefly comment on some of these.

\subsection{Qualitative approximation}
As already pointed out in Remark \ref{rmk:loc_nonloc} the \emph{qualitative} approximation argument does not use any regularizing properties of the underlying (nonlocal) equation. It only exploits the weak unique continuation properties of the fractional Laplacian and is hence a purely nonlocal phenomenon (in the sense that the unique continuation properties of the nonlocal operator determine the approximation properties independently of which additional local contributions are involved in the equation). Provided that the associated problem is well-posed (i.e. that the boundary data are prescribed correctly), it is therefore possible to prove these qualitative approximation properties for general operators of the form $L + (-\D)^s $, where $L$ is an arbitrary \emph{local} differential operator. This recovers (a part of) the result of \cite{DSV16}.

In general, qualitative approximation results which are obtained by means of the Runge 
approximation, require two ingredients:
\begin{itemize}
\item[(a)] well-posedness of the underlying equation and its adjoint,
\item[(b)] weak unique continuation for the associated nonlocal operator.
\end{itemize}
We again emphasize that in (b) only the weak unique continuation properties of the \emph{nonlocal} operator are of relevance.
As the weak unique continuation property is such a crucial ingredient, it is an interesting question to ask for which nonlocal operators it is valid. A large class of operators for which this holds is identified by Isakov: 

\begin{lem}[\cite{Isakov}, Lemma 3.5.4]
\label{lem:Isakov}
Let $\mu_j $, $j\in \{1,2\}$, be measures with $\supp(\mu_j) \subset B_r$. Let $E \in \mathcal{S}'(\R^n)$. Assume that
$\F(E)$ cannot be written as the sum of a meromorphic function (in $\C^n$) and a distribution supported on the zero set of some nontrivial entire function. Then if $E\ast (\mu_1-\mu_2) = 0 $ in $\R^n \setminus B_r$, we have that $\mu_1 = \mu_2$ globally.
\end{lem}

For convenience, we recall the proof of Isakov.
 
\begin{proof}
As $\mu_j$, $j\in\{1,2\}$, and $E\ast (\mu_1 -\mu_2)$ are compactly supported, the Paley-Wiener theorem asserts that
$\F(\mu_1 - \mu_2)$ and $\F(E \ast (\mu_1 -\mu_2))$ are analytic functions. But we have that
\begin{align*}
\F(E\ast (\mu_1 - \mu_2)) = \F(E) \F(\mu_1 - \mu_2).
\end{align*}
Thus, on the set in which the entire function $\F(\mu_1 - \mu_2)$ does not vanish, we have that
\begin{align*}
\F(E) = \frac{\F(E\ast (\mu_1 - \mu_2))}{\F(\mu_1 - \mu_2)}.
\end{align*}
The right hand side is by definition a meromorphic function in $\C^n$ (and thus by \cite{Lojasiewicz} defines an element of $\mathcal{D}'(\R^n)$). 
As a consequence, $\F(E)$ can be written as
\begin{align*}
\F(E) = \frac{\F(E\ast (\mu_1 - \mu_2))}{\F(\mu_1 - \mu_2)} 
+ h,
\end{align*}
where the first term on the right hand side is a meromorphic function, while the second term $h$ is a distribution supported on the zero set of the entire function $\F(\mu_1 - \mu_2)$. This is a contradiction to the assumption of the lemma unless $\F(\mu_1 - \mu_2) = 0$ globally.
\end{proof}

Due to the presence of a branch-cut, Isakov's lemma for instance applies to operators of the form
$L:\mathcal{D}(\R^{m_1}\times \dots \times \R^{m_k}) \rightarrow \mathcal{D}(\R^{m_1}\times \dots \times \R^{m_k})$ given by
\begin{align}
\label{eq:op_nonloc}
\bar{L}= \sum\limits_{j=1}^{k} a_{j}(-\D_{X_j})^{s_j},
\end{align}  
for $X_j \in \R^{m_j}$, $a_j \in \R$ and $s_j \in (0,1)$. In particular, these operators need not be elliptic. We will give the proof for more general operators of the form 
\begin{align*}
\label{eq:op_nonloc_1}
\tilde{L}:=(-\D_{X_1})^{s_1} + m(D_{X_2}), 
\end{align*}
where $m(D_{X_2})$ is a Fourier multiplier in the $X_2$ variable with at most polynomial growth in Fourier space, i.e., there exists $N\in \N$ such that
\begin{align*}
\mathcal{F} (m(D_{X_2})u) = m(\eta) \hat{u}(\eta), \quad |m(\eta)| \leq C(1+|\eta|)^N.
\end{align*}

\begin{cor}
\label{cor:frac1}
Let $s_1 \in (0,1)$, $n_1, n_2 \in \N\cup \{0\}$, $n_1 \geq 1$ and let $m(D_{X_2})$ be a Fourier multiplier.
Let $n = n_1 + n_2$ and $\varphi:\R^n= \R^{n_1}\times \R^{n_2} \rightarrow \R$, $\varphi \in H^{-s}(\R^n)$ for some $s > 0$, be such that for some $r > 0$
\begin{align*}
\supp(((-\D_{X_1})^{s_1} + m(D_{X_2}))  \varphi ) \subset B_{r}, \quad \supp(\varphi) \subset B_{r},
\end{align*}
where $X=(X_1,X_2) \in \R^{n_1}\times \R^{n_2}$.
Then we have that $\varphi =0$.
\end{cor}

\begin{figure}[t]
\includegraphics[scale =0.8]{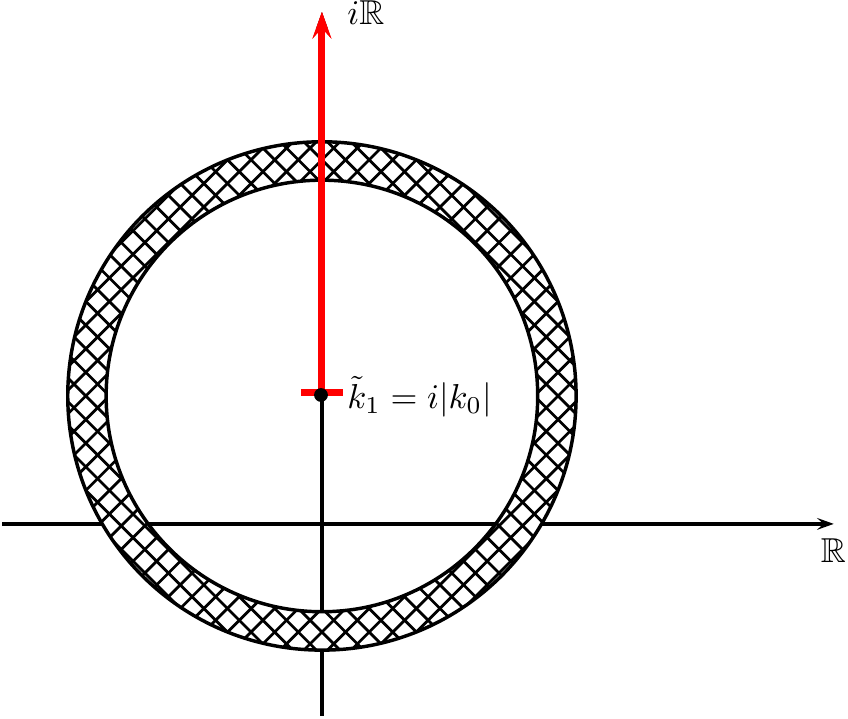}
\caption{The annulus from the proof of Corollary \ref{cor:frac1}. The shaded area corresponds to the annulus $A_{R_1(\tilde{k}_1),R_2(\tilde{k}_1)}(\tilde{k}_1)$ with $\tilde{k}_1$ chosen to be $i|k_0|$. The red line indicates the discontinuity line of the function $e^{s_1 \log(k_1^2 + |k_0|^2)}$.}
\label{fig:complex}
\end{figure}

\begin{proof}
Instead of reducing the corollary to the statement of Lemma \ref{lem:Isakov}, we prove it directly by a similar argument. By virtue of our assumptions and by the Paley-Wiener theorem, we first infer that the functions $\hat{\varphi}(k,\eta)$ and $(|k|^{2 s_1}+ m(\eta)) \hat{\varphi}(k,\eta)$ are real analytic and have entire analytic extensions into $\C^n$. With slight abuse of notation, we do not change the notation for the analytic extensions, i.e., for instance the function $\hat{\varphi}(k,\eta)$ denotes both the original function defined on $\R^n$ and its analytic extension onto $\C^n$ (which of course is consistent by restriction).

Let us next assume that the statement of the corollary were wrong, i.e.\ that $\varphi \not\equiv 0$ as a function on $\R^n$ and hence also $\hat{\varphi}\not\equiv 0$ as a function on $\C^n$. This implies that there exists a vector $\xi'=(k_0,\eta_0) \in \R^{n_1-1}\times \R^{n_2} = \R^{n-1}$ such that $\hat{\varphi}(k_1, \xi')\not\equiv 0$ as a function of $k_1 \in \R$ (and hence also as a function of $k_1 \in \C$). As $\hat{\varphi}$ is analytic as a function in each of its variables, this entails that $\hat{\varphi}(k_1, \xi')$, as a function on $\C$, only has a countable discrete set $Z \subset \C$ of zeroes. In particular, for each $\tilde{k}_1 \in \C$ there exist radii $R_1(\tilde{k}_1) > R_2(\tilde{k}_1)>|\tilde{k}_1|$ such that on the open annulus $A_{R_1(\tilde{k}_1),R_2(\tilde{k}_1)}(\tilde{k}_1):=B_{R_1(\tilde{k}_1)}(\tilde{k}_1)\setminus \overline{B_{R_2}(\tilde{k}_1)}$ centered at $\tilde{k}_1\in \C$ the function $\hat{\varphi}(k_1, \xi')$ does not have any zeroes and such that $A_{R_1(\tilde{k}_1),R_2(\tilde{k}_1)}(\tilde{k}_1) \cap \R$ is a relatively open, nonempty set (else it would be possible to construct an accumulation point of zeroes by considering a decreasing sequence $(R_2^{(j)})$ of radii with $R_2^{(j)} \to R_1$ and by invoking the theorem of Bolzano-Weierstra{\ss}). But for some analytic function $g(\xi)$ with $\xi'=(k_0,\eta_0)$ we have that 
\begin{align*}
(|(k_1,k_0)|^{2 s_1} +m(\eta_0)) \hat{\varphi}(k_1, \xi') = g(k_1,\xi')  \mbox{ on } A_{R_1(\tilde{k}_1),R_2(\tilde{k}_1)}(\tilde{k}_1) \cap \R.
\end{align*}
Therefore, on the one hand, for each $\tilde{k}_1\in \C$ we can define an analytic continuation of the function $f(k_1):=|(k_1,k_0)|^{2 s_1} + m(\eta_0)$ on $A_{R_1(\tilde{k}_1),R_2(\tilde{k}_1)}(\tilde{k}_1)$ by setting $f(k_1)= \frac{g(k_1,\xi')}{\hat{\varphi}(k_1,\xi')}$. This defines a holomorphic function on $A_{R_1(\tilde{k}_1),R_2(\tilde{k}_1)}(\tilde{k}_1)$. On the other hand, for the standard choice of the logarithm (where the branch cut is located on the negative real axis), the function $k_1 \mapsto e^{s_1 \log(k_1^2 + |k_0|^2) } + m(\eta_0)$ is analytic in $\C \setminus \{ \pm i\alpha \,;\, \alpha \geq |k_0| \}$ and hence this function is also obtained by analytic continuation from the restriction of $f(k_1)$ onto $ A_{R_1(\tilde{k}_1),R_2(\tilde{k}_1)}(\tilde{k}_1) \cap \R$.
By uniqueness of the analytic extension we thus deduce that 
\begin{align*}
f(k_1)= e^{s_1 \log(k_1^2 + |k_0|^2) } + m(\eta_0)
\mbox{ on } A_{R_1(\tilde{k}_1),R_2(\tilde{k}_1)}(\tilde{k}_1) \setminus \{ \pm i\alpha \,;\, \alpha \geq |k_0| \}.
\end{align*} 
But as the logarithm is discontinuous at its branch points on $\R_{-}\times\{0\} \subset \C$ and as $s_1 \in (0,1)$, the function $e^{s_1 \log(k_1^2 + |k_0|^2)}$ is discontinuous along the line $i\R_+ + i |k_0| \subset \C$ (c.f.\ Figure \ref{fig:complex}). If we choose $\tilde{k}_1 = i |k_0|$, this yields a contradiction to the analyticity of $f(k_1)$ on $A_{R_1(\tilde{k}_1),R_2(\tilde{k}_1)}(\tilde{k}_1)$.

Thus, the contradiction assumption must have been wrong and hence $\varphi=0$, proving the desired result.
\end{proof}

\begin{rmk}
We remark that technically an important ingredient in our argument was the reduction to the one-dimensional situation, which allowed us to invoke properties of holomorphic functions in a single complex variable instead of working with several complex variables.
\end{rmk}

\begin{rmk}
The requirement $s_1 \in (0,1)$ can be relaxed; all powers $s\in \R$, which ensure the presence of a branch-cut for the continuation of $|\xi|^{2s}$ can be used in the argument from above.
\end{rmk}

As discussed in \cite{Isakov}, Lemma \ref{lem:Isakov} does not only apply to the specific class of nonlocal operators from \eqref{eq:op_nonloc}, but also to other interesting operators.

If the underlying equations are well-posed, the Runge-type arguments from above yield for instance the following Corollary: 

\begin{cor}
\label{cor:Isakov}
Let $s_1\in(0,1)$, $n, n_1, n_2 \in \N$ with $n=n_1+n_2$ and let $B_1 \subset \R^{n_1}\times \R^{n_2}$.
Let $\tilde{L}$ be as in \eqref{eq:op_nonloc} where the Fourier multiplier $m$ is real, i.e. $m(\eta)\in \R$ for all $\eta \in \R^{n_2}$. Assume that for some $s > 0$, $\tilde{L}$ is bounded $H^s(\R^n) \to H^{-s}(\R^n)$, and that the problem 
\[
\tilde{L} h = v \mbox{ in } B_1,\ h = 0 \mbox{ in } \R^n \setminus \overline{B}_1
\]
has a unique solution $h \in H^s(\R^n)$ for any function $v \in L^2(B_1)$. Denote by $P_{\tilde{L}}f: C^{\infty}_c(\R^n \setminus \overline{B}_1) \rightarrow H^s(\R^n)$ the corresponding solution operator to the problem
\[
\tilde{L} u = 0 \mbox{ in } B_1,\ u = f \mbox{ in } \R^n \setminus \overline{B}_1.
\]
Then we have that for any $R>1$ the set
\begin{align*}
\mathcal{R}:=\{u|_{B_1}: u = P_{\tilde{L}}f, \ f \in C_c^{\infty}(\R^n \setminus \overline{B}_R)\}
\end{align*}
is dense in $L^2(B_1)$.
\end{cor}

\begin{proof}
Arguing similarly as in Theorem \ref{prop:approx_qual}, by a Hahn-Banach argument, the density result reduces to the weak unique continuation property of the nonlocal operator $\tilde{L}$ in $\R^n \setminus B_R$. This however follows from Corollary \ref{cor:frac1}. 

More precisely, we show that if $v\in L^2(B_1)$ is such that
$(P_{\tilde{L}}f,v)_{L^2(B_1)}=0$ for all $f\in C_c^{\infty}(\R^n \setminus \overline{B}_R)$, then necessarily $v=0$. Indeed, using the assumed well-posedness, we define $h\in \tilde{H}^s(\R^n)$ by the requirement
\begin{align*}
\tilde{L} h = v \mbox{ in } B_1,\ h = 0 \mbox{ in } \R^n \setminus \overline{B}_1.
\end{align*}
Then, 
\begin{align}
\label{eq:HB_1}
\begin{split}
0  & = (P_{\tilde{L}}f,v)_{L^2(B_1)}
= (P_{\tilde{L}}f, \tilde{L}h)_{L^2(B_1)}
= (P_{\tilde{L}}f, ((-\D_{X_1})^{s_1} + m(D_{X_2})) h)_{L^2(B_1)}\\
&= (P_{\tilde{L}}f - f, ((-\D_{X_1})^{s_1} + m(D_{X_2})) h)_{L^2(\R^n)}
= (((-\D_{X_1})^{s_1} + m(D_{X_2})) (P_{\tilde{L}}f - f),  h)_{L^2(\R^n)}\\
&= -(((-\D_{X_1})^{s_1} + m(D_{X_2}))  f,  h)_{L^2(\R^n)}
= -(  f,  ((-\D_{X_1})^{s_1} + m(D_{X_2})) h)_{L^2(\R^n)}.
\end{split}
\end{align}
Here we used that $P_{\tilde{L}}f - f \in \tilde{H}^s(B_1)$ and $h \in H^{-s}(\R^n)$, that $\tilde{L}$ is bounded $H^s \to H^{-s}$, and that $m(D_{X_2})$ is self-adjoint. As a consequence, we infer that
\begin{align*}
h = 0 \mbox{ in } \R^n \setminus \overline{B}_1,\ ((-\D_{X_1})^{s_1} + m(D_{X_2})) h \in \R^n \setminus \overline{B}_R.
\end{align*}
Corollary \ref{cor:frac1} then implies that $\varphi = 0$, which entails that $v=0$.
\end{proof}

\begin{rmk}
Assuming the validity of the corresponding well-posedness theory and further supposing that the local and nonlocal contributions act in different variables, it is straightforward to extend the statement of Corollary \ref{cor:Isakov} to a combination of local and nonlocal operators. This follows by observing that as in the case of the fractional heat equation, the local terms ``disappear" on the right hand side of the analogue of the duality argument outlined in \eqref{eq:HB_1}. As in the setting of the heat equation, the variables on which the local operators act are then simply treated as parameters in the unique continuation properties of the nonlocal operators.
\end{rmk}

\subsection{Further constant coefficient operators}
In contrast to the discussion on \emph{qualitative} approximation in Section \ref{sec:qual}, the arguments on the \emph{quantitative} approximation in Section \ref{sec:quant} also relied on properties of the underlying operator (including the local terms). Here we made use of two main ingredients: We combined
\begin{itemize} 
\item \emph{quantitative} weak unique continuation properties (where the main thrust originated from the nonlocal part of the operator), 
\item with specific (regularity) properties of the \emph{full underlying operator}, in the form of (global) energy estimates, c.f.\ \eqref{eq:global} and Remark \ref{rmk:quant}. 
\end{itemize}
These properties are for instance reflected in the respective norms of $h$, which arise in the estimate on the cost of approximation (c.f.\ the bounds in Step 3 in the proof of Theorem \ref{prop:cost}). While this entails that in contrast to the \emph{qualitative} approximation properties their \emph{quantitative} counterparts depend more delicately on the structure of the underlying operator -- also on the elliptic/parabolic/hyperbolic nature of the local part of the operator -- the overall strategy of proof is very robust. It can be applied to a large class of equations, including elliptic/parabolic/hyperbolic ones. To illustrate this, we remark that analogous arguments as outlined above with the same energy functional \eqref{eq:functional} (but where $\varphi$ now solves the dual problem for the fractional wave equation) lead to quantitative approximation properties for the fractional wave equation 
\begin{equation}
\label{eq:wave}
\begin{split}
(\p_t^2 + (-\D)^{s})u & = 0 \mbox{ in } B_1 \times (-1,1),\\
u & = f \mbox{ in } (\R^n \setminus \overline{B}_1) \times (-1,1),\\
u & = f, \ \p_t u = \p_t f \mbox{ on }  \R^n \times \{-1\}.
\end{split}
\end{equation}
Here $W\subset \R^n \setminus \overline{B}_1$ is a bounded Lipschitz set.
Arguing by a Galerkin approximation, this problem is well-posed. We consider the Poisson operator for \eqref{eq:wave}, 
\begin{align*}
P_s^w:L^2((-1,1), C^{\infty}_c(W)) \to L^2(B_1 \times (-1,1)), \ \ f \mapsto P_s^w f=u|_{B_1 \times (-1,1)}.
\end{align*}
In the setting of the wave equation the energy estimates replacing \eqref{eq:global} become
\begin{align*}
\sup\limits_{t\in[-1,1]} \left( \|\p_t \varphi\|_{L^2(B_1)} + \| \varphi\|_{H^s(B_1)} \right) \leq C \|v\|_{L^2((-1,1),L^2(B_1))}.
\end{align*}

As a consequence of the arguments leading to Theorem \ref{prop:cost}, we also infer a result on the cost of approximation for the wave equation:

\begin{thm}[Cost of approximation for the wave equation]
\label{prop:cost_1}
Let $h\in H^2_0(B_1 \times (-1,1))$ and $\epsilon>0$. Let $W \subset \R^n \setminus \overline{B}_1$ be a bounded Lipschitz domain with $\overline{W} \cap \overline{B}_1 = \emptyset$.
Then there exists a control function $f\in L^2((-1,1), C^{\infty}_c(W))$ such that
\begin{equation*}
\begin{split}
&\|h-P_s^w f\|_{L^2(B_1 \times (-1,1))} \leq \epsilon,\\ 
&\|f\|_{L^2(W \times (-1,1))}\leq C e^{C (1+\|h\|_{H^2(B_1 \times (-1,1))}^{\sigma})\epsilon^{-\sigma}}\|h\|_{H^2(B_1 \times (-1,1))},
\end{split}
\end{equation*}
for constants $C>1$ and $\sigma > 0$, which only depend on $n$, $s$, and $W$.
Moreover, we note that $f$ can be expressed in terms of the minimizer of the functional \eqref{eq:functional}.  
\end{thm}

\subsection{Variable coefficient operators}
Last but not least, we emphasize that the described techniques permit us to deal with variable coefficient perturbations of the local and nonlocal parts of the operator (c.f.\ also the recent article \cite{GLX17} for qualitative statements). Here the variable coefficient nonlocal operators can for instance be understood as in  \cite{ST10}, \cite{CS16}. As in \cite{Rue15}, Section 6, and \cite{Rue17}, Section 4, the corresponding estimates carry over to this regime, if the coefficients are suitably regular (c.f.\ Section 6 in \cite{Rue15} or also \cite{Yu16} for weak and strong unique continuation properties of the variable coefficient fractional Laplacian and the associated necessary regularity assumptions on the coefficients).

For simplicity we only discuss the simplest possible extensions. Operators which for instance involve lower order contributions can also be dealt with in this framework. 
Let $L = \p_{n+1} x_{n+1}^{1-2s}\p_{n+1} + x_{n+1}^{1-2s} \p_{i} a^{ij}\p_j$,
where $i,j \in \{1,\dots,n\}$ and $a^{ij}:\R^{n} \rightarrow \R^{n\times n}_{sym}$ is a positive definite, symmetric, Lipschitz continuous matrix field. Then, following \cite{CS07}, \cite{ST10}, \cite{CS16}, we define
\begin{align*}
(-\D_{a^{ij}})^{s}u(x') := \lim\limits_{x_{n+1}\rightarrow 0} x_{n+1}^{1-2s}\p_{n+1}\bar{u}(x',x_{n+1}),
\end{align*}
where $\bar{u}$ solves the equation
\begin{align*}
L \bar{u} &= 0 \mbox{ in } \R^{n+1}_+,\ \bar{u} = u \mbox{ on } \R^{n} \times \{0\}. 
\end{align*}
The operator $(-\D_{a^{ij}})^{s}$ is self-adjoint.
Considering the problem
\begin{equation*}
\begin{split}
(\p_t +(-\D_{a^{ij}})^{s}) u &= 0 \mbox{ in } B_1 \times (-1,1),\\
 u & = f \mbox{ in } (\R^n \setminus \overline{B}_1) \times (-1,1),\\
 u & = f \mbox{ in } \R^n \times \{-1\},
\end{split}
\end{equation*}
recalling (a slight modification of) the well-posedness theory for the mixed Dirichlet-Neumann problem from \cite{KRSIV} and denoting the corresponding Poisson operator by $P_{s,a^{ij}}$, we obtain the direct analogue of Theorem \ref{prop:cost}:

\begin{thm}[Cost of approximation for variable coefficients]
\label{prop:cost_2}
Let $h\in H^1_0(B_1 \times (-1,1))$ and $\epsilon>0$. Let $W \subset \R^n \setminus \overline{B}_1$ be a Lipschitz domain with $\overline{W} \cap \overline{B}_1 = \emptyset$.
Then there exists a control function $f\in L^2(W \times (-1,1))$ such that
\begin{equation*}
\begin{split}
&\|h-P_{s,a^{ij}} f\|_{L^2(B_1 \times (-1,1))} \leq \epsilon,\\ 
&\|f\|_{L^2(W \times (-1,1))}\leq C e^{C (1+\|h\|_{H^1(B_1 \times (-1,1))}^{\sigma})\epsilon^{-\sigma}}\|h\|_{H^1(B_1 \times (-1,1))},
\end{split}
\end{equation*}
where $C > 1$ and $\sigma > 0$ depend on $n$, $s$, $W$, and the Lipschitz norm of $a^{ij}$.
\end{thm}

\begin{proof}
We only give a sketch of the argument, as there are no major changes with respect to the proof of Theorem \ref{prop:cost}.
For the qualitative approximation property, it suffices to note that the crucial identity
\begin{align*}
(v,P_{s,a^{ij}}f)_{L^2(B_1 \times (-1,1))} = - ((-\D_{a^{ij}})^s \varphi, f)_{L^2(\R^n)}
\end{align*}
remains valid. This can for instance be inferred by the extension definition of the operator.

Next we note that the quantitative propagation of smallness result which is based on three balls and boundary-bulk interpolation arguments is also true in this set-up. This then allows to argue variationally as previously. Here we consider the functional
\begin{align}
\label{eq:functional_v}
\Je(v) = \frac{1}{2}\int\limits_{W \times (-1,1)}|\eta(-\D_{a^{ij}})^{s} \varphi|^2 \,dx + \epsilon \|v\|_{L^2(B_1 \times (-1,1))} - \int\limits_{B_1 \times (-1,1)} h v \,dx \,dt,
\end{align}
where $\varphi$ and $v$ are related through
\begin{equation}
\label{eq:eq_main_adj}
\begin{split}
(-\p_t +(-\D_{a^{ij}})^{s}) \varphi &= v \mbox{ in } B_1 \times (-1,1),\\
 \varphi & = 0\mbox{ in } (\R^n \setminus \overline{B}_1) \times (-1,1) ,\\
 \varphi & = 0 \mbox{ in } \R^n \times \{1\}.
\end{split}
\end{equation}
This then concludes the argument.
\end{proof}

\bibliographystyle{alpha}
\bibliography{citations_frac}

\end{document}